\documentclass[11pt]{amsart}
\usepackage{comment}
\usepackage{graphicx}
\usepackage{amsfonts}
\usepackage{amsthm}
\usepackage{amsmath}
\usepackage{amssymb}
\usepackage{hyperref}
\usepackage[arrow,matrix,curve,color]{xy}
\usepackage[margin=1in]{geometry}
\usepackage{pb-diagram,pb-xy}
\usepackage{hyperref}   
\usepackage{enumitem}
\usepackage[normalem]{ulem}
\usepackage{color}
\usepackage[normalem]{ulem}
\usepackage{tikz}
\usepackage{xspace}
\definecolor{light-blue}{rgb}{0.8,0.85,1}
\definecolor{light-red}{rgb}{1,.4,.4}
\definecolor{purp}{rgb}{.7,.3,1}
\definecolor{yel}{rgb}{1,1,.5}
\definecolor{cy}{rgb}{0,1,1}
\usepackage{colortbl}
\usepackage{multirow}
\usepackage{array}

\theoremstyle{plain}
\newtheorem{theorem}{Theorem}[section]
\newtheorem{corollary}[theorem]{Corollary}
\newtheorem{lemma}[theorem]{Lemma}
\newtheorem{proposition}[theorem]{Proposition}
\newtheorem*{thmA}{Theorem A}
\newtheorem*{thmB}{Theorem B}

\theoremstyle{definition}
\newtheorem{remark}[theorem]{Remark}
\newtheorem{definition}[theorem]{Definition}
\newtheorem{example}[theorem]{Example}

\newcommand{\CP}{\mathbb{CP}}
\newcommand{\HP}{\mathbb{HP}}
\newcommand{\p}{\partial}
\newcommand{\fb}{\mbox{-}\mathrm{fb}}
\newcommand{\spin}{\mathsf{spin}}
\newcommand{\spinc}{\mathsf{spin}^c}
\newcommand{\spinpsc}{\mathsf{spin, psc}}
\newcommand{\MSpin}{\mathsf{MSpin}}
\newcommand{\MSpinc}{\mathsf{MSpin}^c}

\newcommand{\MU}{\mathsf{MU}}
\newcommand{\KO}{\mathsf{KO}}
\newcommand{\K}{\mathsf{K}}
\newcommand{\ko}{\mathsf{ko}}
\newcommand{\KU}{\mathsf{KU}}
\newcommand{\ku}{\mathsf{ku}}

\newcommand{\sH}{\mathsf{H}}
\newcommand{\psc}{positive scalar curvature}
\newcommand{\co}{\colon\,}

\newcommand{\bR}{\mathbb R}
\newcommand{\bC}{\mathbb C}

\newcommand{\bF}{\mathbb F}
\newcommand{\bH}{\mathbb H}

\newcommand{\bZ}{\mathbb Z}

\newcommand{\bP}{\mathbb P}

\newcommand{\bS}{\mathbb S}
\newcommand{\cA}{\mathcal A}

\newcommand{\cC}{\mathcal C}

\newcommand{\cL}{\mathcal L}

\newcommand{\cO}{\mathcal O}

\newcommand{\cR}{\mathcal R}

\newcommand{\SO}{\mathop{\rm SO}}
\newcommand{\SU}{\mathop{\rm SU}}
\newcommand{\U}{\mathrm{U}}

\newcommand{\Sp}{\mathop{\rm Sp}}

\newcommand{\PSp}{\mathop{\rm PSp}}

\newcommand{\tN}{\widetilde N}

\newcommand{\tf}{\widetilde f}

\newcommand{\pt}{\text{\textup{pt}}}
\newcommand{\lp}{\textup{(}}
\newcommand{\rp}{\textup{)}}

\newcommand{\Ca}{$C^*$-algebra}
\newcommand{\ind}{\operatorname{ind}}

\newcommand{\Hom}{\operatorname{Hom}}

\newcommand{\im}{\operatorname{Im}}

\newcommand{\Sq}{\operatorname{Sq}}

\newcommand{\image}{\operatorname{image}}
\newcommand{\assemb}{\mathsf{assemb}}

\newcommand{\Dirac}{\partial\!\!\!/}

\newcommand{\spS}{{{\spin,\, S^1\fb}}}
\newcommand{\spZk}{{{\spin,\, \bZ/k\fb}}}
\newcommand{\Zk}{{{\bZ/k\fb}}}
\newcommand{\bord}{\rightsquigarrow} 
\DeclareMathOperator{\Td}{Td}
\DeclareMathOperator{\ch}{ch}
\DeclareMathOperator{\Sym}{Sym}
\newcommand{\Ahat}{\widehat{A}}
\newcommand{\Rgen}{R^{\text{gen}}}
\newcommand{\BSpinc}{B\mathrm{Spin}^c}
\title[Manifolds with fibered singularities]{Positive scalar curvature on
  manifolds\\ with fibered singularities}
\author{Boris Botvinnik}
\address{Department of Mathematics\\
University of Oregon\\
Eugene OR 97403-1222, USA} 
\email[Boris Botvinnik]{botvinn@uoregon.edu}
\thanks{BB was partially supported by Simons collaboration grant 708183}
\urladdr{http://pages.uoregon.edu/botvinn/}
\author{Jonathan Rosenberg}
\address{Department of Mathematics\\
University of Maryland\\
College Park, MD 20742-4015, USA} 
\email[Jonathan Rosenberg]{jmr@math.umd.edu}
\urladdr{http://www2.math.umd.edu/\raisebox{-3pt}{~}jmr/}
\thanks{JR Partially supported by {U.S.} NSF grant number DMS-1607162.}
  
\begin{document}
\begin{abstract}
A {\lp}compact{\rp}
manifold with fibered $P$-singularities is a {\lp}possibly{\rp}
singular pseudomanifold $M_\Sigma$ with two strata: an open nonsingular stratum
$\mathring M$ {\lp}a smooth open manifold{\rp}
and a closed stratum $\beta M$
{\lp}a closed manifold of positive codimension{\rp}, such that a tubular
neighborhood of $\beta M$ is a fiber bundle with fibers each looking like
the cone on a fixed closed manifold $P$. We discuss what it means for
such an $M_{\Sigma}$ with fibered $P$-singularities to
admit an appropriate Riemannian metric of positive scalar
curvature, and we give necessary and sufficient conditions
{\lp}the necessary conditions
based on suitable versions of index theory, the sufficient
conditions based on surgery methods and homotopy theory{\rp}
for this to happen when the singularity type $P$ is either $\bZ/k$ or
$S^1$, and $M$ and the boundary of the tubular neighborhood of
the singular stratum are simply connected and carry spin structures. 
Along the way, we prove some results of perhaps independent interest,
concerning metrics on spin$^c$ manifolds with positive
``twisted scalar curvature,'' where the twisting comes from the
curvature of the spin$^c$ line bundle.
\end{abstract}

\maketitle

\section{Introduction}
\label{sec:intro}
\subsection{Manifolds with fibered singularities}
In the paper \cite{MR1857524}, one of us studied the problem of when a
spin manifold with Baas-type singularities admits a metric of
\psc. This was done when the singularities are a combination of the
types $\bZ/2$, $\eta =$ a circle $S^1$ equipped with
the non-bounding spin
structure, and the Bott manifold of dimension $8$
(a geometric generator of Bott periodicity in $KO$-homology). We will
use the abbreviation \emph{psc-metric} for \emph{metric of \psc}.
\vspace{2mm}

This paper considers a similar problem of existence of positive scalar
curvature metrics for spin manifolds with \emph{fibered
  $P$-singularities}, where $P$ is a closed manifold. We take $P$ to
be a compact Lie group, either a cyclic group $\bZ/k$ for some $k$, or
$S^1$.  In a sequel paper \cite{BJP}, we study the case where $P$
is a compact semisimple Lie group, such as $\SO(3)$ or $\SU(2)$.
Another sequel \cite{BJ1} goes into more detail about the cases
$P=\bZ/2$ and $P=S^1$.
\vspace{2mm}

In more detail, let $P$ be a closed manifold, and $M$ be a
compact manifold with boundary $\p M$. We assume that the boundary
$\p M$ is the total space of a smooth fiber bundle $p\co  \p M \to \beta M$
with the fiber $P$. We denote by $N (\beta M)$ the total space of
the associated fiber bundle $\p M\times_P C(P) \to \beta M$, where the
manifold $P$ is replaced by the cone $C(P)$ fiberwise.
The notation is a reminder that this will be a tubular neighborhood
of the stratum $\beta M$.
The associated singular space $M_\Sigma:= M\cup_{\p M} - N(\beta M)$
has a singular stratum $\beta M$ whose normal bundle has fibers
homeomorphic to the cone on $P$, but where the normal bundle is
not necessarily a trivial bundle.
%
%
We call $M_{\Sigma}$ a \emph{manifold with fibered
  $P$-singularity}. Clearly the original manifold $M$ with the given
smooth fibration $p\co \p M \to \beta M$ uniquely determines the singular
space $M_{\Sigma}$.  It's traditional to call $\beta M$ the \emph{Bockstein}
and $P$ the \emph{link}.

\vspace{2mm}

The situation we study fits nicely into a more general framework.  The
compact metrizable space $M_{\Sigma}$ is taken to be a Thom-Mather
stratified space (see \cite{MR2958928}) of depth one.  That means that
$M_{\Sigma}$ is the union of a smooth closed manifold $\beta M$ (the
singular stratum) and an open smooth manifold (the regular stratum)
$M_\Sigma^{\text{reg}}=M_\Sigma\smallsetminus \beta M$.  There is an
open neighborhood $N$ of $\beta M$ in $M_\Sigma$, with a continuous
retraction $\pi\co N \to \beta M$ and a continuous map
$\rho\co N\to [0, \infty)$ such that $\rho^{-1}(0) = \beta M$,
where $\pi\co N\to \beta M$ is a fiber bundle over
$\beta M$ with fiber $C(P)$, the cone
over $P$, as above.  The original manifold $M$, called a
\emph{resolution of} $M_{\Sigma}$, can be identified with $M_\Sigma
\smallsetminus \rho^{-1}([0, 1))$, so that the boundary $\p M =
\rho^{-1}(1)$ is the total space of a fibration over $\beta M$ with
typical fiber $P$. Clearly the interior of $M$ can be identified with
the regular stratum $M_\Sigma^{\text{reg}}$. Conversely,
given a compact manifold $M$ with fibered boundary
$P\to\p M\to\beta M$, we obtain a Thom-Mather stratified
space with two strata by collapsing the fibers.

%
%
\subsection{Riemannian metrics on manifolds with fibered singularities}
There are (at least) two natural definitions of a Riemannian
metric on such a singular manifold $M_\Sigma$.  The first
possibility, which one can call a \emph{cylindrical metric},
the definition used in \cite{MR1857524}, is a Riemannian
metric in the usual sense on the nonsingular manifold $M$, which
is a product metric $dr^2 + g_{\p M}$ on a collar neighborhood on
the boundary $\p M$, such that the compact Lie group $P$ acts freely by
isometries of the boundary and the map $\partial
M\xrightarrow{p}\beta M$ is a Riemannian submersion.  The
curvature of such a metric is defined as usual on the
(nonsingular) manifold $M$.  Such a metric also determines a
Riemannian metric on the ``Bockstein'' manifold $\beta M$.

\vspace{2mm}

Assume $M_{\Sigma}$ has a psc-metric $g$ as above. Since the metric
$g$ is assumed to be a product metric in a collar neighborhood of the
boundary $\p M$, the restriction $g|_{\partial M}$ is a $P$-invariant
psc-metric. Thus the question of existence of a cylindrical psc-metric
comes down to whether or not there exists a $P$-invariant psc-metric on
$\p M$ that extends to $M$.  When $P = \SU(2)$ or $\SO(3)$, the boundary
$\partial M$ always has a $P$-invariant psc-metric by
\cite{MR0358841}, for which the quotient map is a Riemannian
submersion, so the only question is whether or not this metric on
$\partial M$ extends to $M$. By contrast, when $P = \bZ/k$, the map
$\partial M\to \beta M$ is a covering map, and $\partial M$ has a
$P$-invariant psc-metric if and only if $\beta M$ has a psc-metric,
for which there is a well-developed obstruction theory (e.g.,
\cite{MR866507,MR842428,MR1133900,MR2093079}). When $P = S^1$, the
same is true; i.e., $\partial M$ has an $S^1$-invariant psc-metric
if and only if $\beta M$ has a psc-metric, but this is now a hard
theorem of B\'erard-Bergery \cite[Theorem C]{BB}.

\vspace{2mm}

%
The problem with the notion of cylindrical metric is that it doesn't
really take into account the local structure near $\beta M$.  A second
possible definition is what one could call a \emph{conical metric} on
$M_\Sigma$.  In this point of view, the primary object of study is the
singular manifold $M_\Sigma$, not the manifold $M$ with non-empty
boundary. A conical metric is again an ordinary Riemannian metric on
the nonsingular part of $M_\Sigma$ (which is diffeomorphic to the
interior of $M$), but we require its local behavior near the singular
stratum to look like $dr^2 + r^2g_P + p^*g_{\beta M}$, where $g_P$ is
a translation-invariant (standard) metric on $P$, $g_{\beta M}$ is a
metric on the singular stratum $\beta M$, and $r$ is the distance to
the singular stratum. Note that such a metric on the nonsingular
stratum is necessarily incomplete, with $M_\Sigma$ its metric completion.
We assume that near the boundary of
the tubular neighborhood $N(\beta M)$, the metric transitions to a
cylindrical psc-metric on $(M,\p M)$, in a sense that will be made precise
in Definition \ref{def:conicalmetric}, and thus \emph{existence of a
  conical psc-metric is a stronger requirement than
  existence of a cylindrical psc-metric}. When $P = S^1$ or
$\SU(2)=S^3$, the cone on $P$ is homeomorphic to $\bR^2$ or $\bR^4$,
and so a conical metric near the singular stratum $\Sigma$ locally
looks like an ordinary rotationally invariant Riemannian metric on a
vector bundle over $\beta M$, of fiber dimension $2$, resp., $4$, and
actually extends to a smooth Riemannian metric on $M_\Sigma$.  Here is
our main question:

\vspace{2mm}

\noindent
{\bf Question:} \emph{When does $M_\Sigma$ admit a cylindrical or
  conical psc-metric?}
\label{q:main}
\vspace{2mm}

\begin{remark}
In general, existence of a conical psc-metric is stronger
than existence of a cylindrical psc-metric. However, they are the
same provided $P=\bZ/k$, 
since then the term $r^2g_P$ drops out of the definition of
conical metric.  When $P=S^1$, there are some cases where
existence of a cylindrical psc-metric implies existence of a
conical psc-metric, but for the most part we will focus on the latter,
which has greater geometric significance.  See Remark \ref{rem:cylvscone}
for further discussion.
\end{remark}

\subsection{Main results}
We consider two cases: when $P=\bZ/k$ or $P=S^1$.

\subsubsection{The case of $P=\bZ/k$} We denote by 
$\Omega^{\spZk}_*(-)$ the bordism theory of spin manifolds with fibered
$\bZ/k$-singularities, and by $\MSpin^{\bZ/k\fb}$ the spectrum which
represents this bordism theory.
\vspace{2mm}

In more detail, the group $\Omega^{\spZk}_n(X)$ consists of
equivalence classes of maps $f \co M_{\Sigma} \to X$, where
$M_{\Sigma}$ is the singular space associated to an $n$-dimensional
spin manifold $M$ with fibered $\bZ/k$-singularities (with given
$\bZ/k$-fold regular covering map $p\co \p M \to \beta M$ preserving
the spin structure on $\p M$ induced from the spin structure on $M$).
Two such maps $f\co M_\Sigma\to X$ and $f'\co M'_\Sigma\to X$ are said
to be equivalent if there is a spin bordism
$\overline M \co M \rightsquigarrow M'$ between $M$ and
$M'$ as spin manifolds with boundary, with a $k$-fold regular covering
map $\overline p \co \p \overline M \to \beta \overline M$ given by a
free action of $\bZ/k$ on $ \p \overline M$, such that the
restrictions $\overline p|_{\p M}$ and $\overline p|_{\p M'}$ coincide
with the corresponding maps $p \co \p M \to \beta M$ and $p' \co \p M'
\to \beta M'$, and there is a map $\overline f\co \overline M \to X$ restricting
to $f$ and $f'$ on $M$ and $M'$.  In particular, the manifold $\beta
\overline M$ gives a spin bordism of regular closed spin manifolds
$\beta \overline M \co \beta M \rightsquigarrow \beta M'$.
\vspace{2mm}

The groups $\Omega^{\spZk}_n$ are closely related to the regular spin
bordism groups; indeed, there is an exact triangle
\begin{equation}\label{eq:1}
  \begin{diagram}
    \setlength{\dgARROWLENGTH}{1.95em}
  \node{\Omega^{\spin}_*}
          \arrow[2]{e,t}{i}
  \node[2]{\Omega^\spZk_*}
          \arrow{sw,t}{\beta}
  \\
  \node[2]{\Omega^{\spin}_*(B\bZ/k)}
  \arrow{nw,t}{\tau}
  \end{diagram}
\end{equation}
Here $i\co  \Omega^{\spin}_*\to \Omega^\spZk_*$ is the homomorphism which
considers a regular spin manifold as a manifold with empty
singularity, the degree $-1$ homomorphism $\beta \co 
\Omega^\spZk_*\to \Omega^{\spin}_{*-1}(B\bZ/k)$ takes a
$\bZ/k$-fibered manifold $M$ to the map $\beta M \to B\bZ/k$
classifying the $\bZ/k$-fold regular covering $p\co  \p M \to \beta M$,
and finally, $\tau \co  \Omega^{\spin}_*(B\bZ/k)\to \Omega^{\spin}_*$ is
a standard transfer.
\vspace{2mm}

We denote by $\KO$ the spectrum representing real $K$-theory, and by
$\alpha \co  \MSpin \to \KO$ the map of spectra corresponding to the index map
$\alpha \co  \Omega^{\spin}_*\to KO_*$. The transfer map $\tau$ in $KO$ gives
the map of spectra $\tau^{\KO}\co \KO \wedge (B\bZ/k)_+ \to \KO$;
we denote by $\KO^{\Zk}$ the cofiber of the map $\tau^{\KO}$.
Then there is a natural map $\alpha^{\Zk}\co  \MSpin^{\Zk}\to \KO^{\Zk}$
which makes the following diagram of spectra commute:
\begin{equation}\label{eq:2}
  \begin{diagram}
    \setlength{\dgARROWLENGTH}{1.95em}
\node{\MSpin\wedge(B\bZ/k)_+}
        \arrow{e,t}{\tau}
        \arrow{s,t}{\alpha\wedge\text{Id}}
\node{\MSpin}
        \arrow{e,t}{i}
        \arrow{s,t}{\alpha}
\node{\MSpin^{\Zk}}
        \arrow{s,t}{\alpha^{\Zk}}
\\
\node{\KO\wedge(B\bZ/k)_+}
        \arrow{e,t}{\tau^{\KO}}
\node{\KO}
        \arrow{e,t}{i^{\KO}}
\node{\KO^{\Zk}}
  \end{diagram}.
\end{equation}
It turns out that the map $\alpha^{\Zk}$ is still
not quite the right ``index map'', since it can be nonzero on some
psc-manifolds with fibered $\bZ/k$-singularities.  For example,
even-dimensional disks with a free $\bZ/k$-action on the boundary
sphere will map nontrivially under
$\alpha^{\Zk}$. However, by composing $\alpha^{\Zk}$
  with the real assembly map $KO_{*-1}(B\bZ/k)\to   KO_{*-1}(\bR[\bZ/k])$
  and its splitting we get the index homomorphism
\begin{equation*}
  \ind^{\Zk}\co
  \Omega^\spZk_*\to KO^\Zk_*.
\end{equation*}
(See Definition \ref{def:Zkobstrmap} and the comments just before it
  for more details.)  The image contains the torsion (all of order
  $2$) in the $KO$-theory of the real group ring of the cyclic group
  $\bZ/k$.  Here is our first main result on the existence of
  psc-metrics:
\begin{thmA}
Let $M$ be a spin manifold with fibered $\bZ/k$-singularities, of
dimension $n\ge 6$.  Assume that $\partial M$ is non-empty, and both
$M$ and $\p M$ are connected and simply connected, and the action of
$\bZ/k$ on $\p M$ preserves the spin structure.  Then $M$ admits a
metric of \psc\ if and only if $\ind^{\Zk}([M])$ vanishes in the group
$KO^\Zk_*$.
\end{thmA}
\begin{proof}[Outline]
Here is an outline of the proof.  In order to prove necessity of
vanishing of $\ind^{\Zk}([M])$, we use {\Ca}ic index theory to show
that $\ind^{\Zk}([M])$ is indeed an 
obstruction to the existence of a psc-metric on a spin manifold $M$
with fibered $\bZ/k$-singularities.  On the other hand, we show also
that the existence of a psc-metric on $M$ depends only on the corresponding
bordism class $[M]\in \Omega^\spZk_*$. To prove that vanishing of the
index $\ind^{\Zk}([M])$ is sufficient, we analyze the spectra
$\MSpin^{\Zk}$ and $\KO^{\Zk}$. In particular, we construct
the cofiber sequences
\begin{equation}\label{eq:3}
  \left\{
  \begin{array}{lclcl}
  \MSpin(\bZ/k) &\to& \MSpin^{\Zk}&\to& \Sigma (\MSpin \wedge B\bZ/k),
  \\
  \KO(\bZ/k)&\to &\KO^{\Zk} &\to&  \Sigma (\KO \wedge B\bZ/k),
  \end{array}\right.
\end{equation}  
where $\MSpin(\bZ/k)$ and $\KO(\bZ/k)$ denote $\MSpin$ and $\KO$ with
$\bZ/k$ coefficients respectively. Moreover, we show that the map
$\alpha^{\Zk}$ is consistent with these decompositions.
\vspace{2mm}
 
To finish the proof, we use the transfer map $T_{\bullet}\co
\Omega^\spZk_{*}(BG)\to \Omega^\spZk_{*+8}$ from \cite{MR1189863},
where $G=\PSp(3)$. Recall that $G$ is the isometry group of the
standard metric on $\HP^2$, and $T_{\bullet}$ takes a map $f\co B \to BG$
to the total space $E$ of the geometric
$\HP^2$-bundle $E\to B$ induced by $f$.
\vspace{2mm}

The diagram \eqref{eq:2} allows us to use \eqref{eq:3}
and the transfer map $T_{\bullet}$ to show that all elements of the
kernel $\ker \ind^{\Zk} \subset \Omega^\spZk_{n}$ of the index map
can be represented by a manifold with fibered $\bZ/k$-singularities
carrying a psc-metric.
\end{proof}
\subsubsection{The case $P=S^1$}
This case in some ways is similar to the case of
$\bZ/k$-singularities; however, it has new interesting features.
\vspace{2mm}

Let $M$ be a manifold with fibered $S^1$-singularities, i.e., $M$
comes with a free $S^1$-action on the boundary $\p M$. Let
$\p M \to \beta M$ be the smooth $S^1$-fiber bundle given by this action.
Since the cone on $S^1$ is $\bR^2$, the corresponding pseudomanifold
$M_\Sigma$ is actually a smooth manifold, but with a distinguished
codimension two submanifold $\beta M$.  There are two separate
problems to consider: the existence of a cylindrical psc-metric, which
is analogous to a psc-metric on a manifold with Baas-type
singularities, or the existence of a conical psc-metric, which is about
the pseudomanifold $M_\Sigma$, but we focus primarily on the latter.
\vspace{2mm}

We assume that $M$ is a spin manifold. However, it is worth pointing out
that under our definition of fibered $S^1$-singularities, even if $M$
is spin, $\beta M$ or $M_\Sigma$ may not be.  For example, suppose
$M=D^{2n}$ is the unit disk in $\bC^n$, and we equip $\partial M =
S^{2n-1}$ with the usual free action of $S^1$ by scalar multiplication
by complex numbers of absolute value $1$.  Then $\beta M = \partial
M/S^1 = \bC\bP^{n-1}$ is non-spin if $n$ is odd, and $M_\Sigma =
\bC\bP^{n}$ is non-spin if $n$ is even. This phenomenon is related
to the dichotomy between ``even'' and ``odd'' actions of $S^1$ on
a spin manifold, discussed in detail in \cite{MR3449263}.
\vspace{2mm}

In general, the $S^1$-fiber bundle $p\co  \p M \to \beta M$ is given by
some classifying map $f\co  \beta M\to \CP^{\infty}$, which induces a
complex line bundle $L \to \beta M$. 
We split $\p M$ into the disjoint
union of its path-components: $\p M= \bigsqcup \p_i M$, and
let $\beta_i M=p(\p_i M)$. Then
for each component $\p_i M$ in $\pi_0(\p M)$, we have two possibilities:
\begin{enumerate}
\item[(i)] the action of $S^1$ on $\p_i M$ is of even type, i.e.,
  spin-structure preserving, or
\item[(ii)] the action of $S^1$ on $\p_i M$ is of odd type, i.e., is
  not spin-structure preserving.
\end{enumerate}  
In the case (i), the spin structure on $\p_i M$ descends to a spin
structure on $\beta_i M$, but the vertical tangent bundle on
$L$ is not spin, since the $S^1$-invariant spin structure on
$S^1$ does not extend to a spin structure on $D^2$ or on $\bC$.
Hence $M_\Sigma$ is spin$^c$ but not spin.
In the case (ii), $w_2(\beta_i M)$ has to
be in the kernel of the homomorphism
$p^* \co  H^2(\beta_i M; \bZ/2) \to H^2(\p_i M; \bZ/2)$,
since $\p M $ is spin. This means that $c_1(L|_{\beta_i M})\equiv
w_2(\beta_i M)$ mod 2 and $L|_{\beta_i M}$ determines a
spin$^c$-structure on $\beta_i M$. Thus the fiber bundle
$p \co  \p M \to \beta M$ always splits into even and odd components:
\[
p^{\text{even}}\co  \p^{\text{even}} M\to \beta M^{(\spin)},
\ \ \ p^{\text{odd}}\co  \p^{\text{odd}} M\to \beta M^{(\spin^c)}.
\]
Thus the Bockstein operator $\beta$ can be described as
\begin{equation}\label{eq:4}
\beta \co  M \mapsto \{(\beta M^{(\spin)}, f|_{M^{(\spin)}}),
(\beta M^{(\spin^c)}, f|_{M^{(\spin^c)}})\}.
\end{equation}
We denote by $\Omega^{\spS}_*(-)$ the bordism theory of spin manifolds
with fibered $S^1$-singularities and by $\MSpin^{S^1\fb}$ the
corresponding spectrum which represents this bordism theory; see
section \ref{sec:S1} for more details.
The Bockstein operator (\ref{eq:4}) induces a homomorphism
\begin{equation*}
  \beta \co \Omega^\spS_n\to
  \Omega^{\spin}_{n-2}(\CP^{\infty})\oplus \Omega^{\spin^c}_{n-2}.
\end{equation*}
Then we have a transfer map
\[
\tau \co  \Omega^{\spin}_{n-2}(\CP^{\infty})\oplus
\Omega^{\spin^c}_{n-2} \to \Omega_{n-1}^{\spin}.
\]
The transfer $\tau$ takes a pair 
$(N,L)\in \Omega^{\spin}_{n-2}(\CP^{\infty})$ or $(N,L)\in \Omega_{n-2}^{\spin^c}$
(in second case $L$ is a spin$^c$-structure on $N$) to the total
space $\widetilde N$ of the corresponding $S^1$-fiber bundle
$\widetilde N \to N$. We show that there is an exact triangle
\begin{equation}\label{eq:5}
  \begin{diagram}
    \setlength{\dgARROWLENGTH}{1.95em}
  \node{\Omega^{\spin}_*}
          \arrow[2]{e,t}{i}
  \node[2]{\Omega^{\spS}_*}
          \arrow{sw,t}{\beta}
  \\
  \node[2]{\Omega^{\spin}_{*}(\CP^{\infty})\oplus
\Omega^{\spin^c}_{*}}
  \arrow{nw,t}{\tau}
  \end{diagram}
\end{equation}
where $i \co  \Omega^{\spin}_*\to \Omega^{\spS}_*$ takes a closed 
spin-manifold $M$ to a manifold with empty fibered $S^1$-singularity.
At the level of spectra, we obtain the following cofibration:
\[
\Sigma (\MSpin\wedge \CP^{\infty}_+)\vee \Sigma \MSpin^c \xrightarrow{\tau} 
\MSpin \xrightarrow{i} \MSpin^{S^1\fb}.
\]
Just as in the $\bZ/k$-case, we consider corresponding
$K$-theories and obtain the following commutative diagram of spectra:
\begin{equation}\label{eq:6}
  \begin{diagram}
    \setlength{\dgARROWLENGTH}{1.95em}
\node{\Sigma (\MSpin\wedge \CP^{\infty}_+)\vee \Sigma \MSpin^c}
        \arrow{e,t}{\tau}
        \arrow{s,t}{\Sigma \alpha\vee \Sigma \alpha^c}
\node{\MSpin}
        \arrow{e,t}{i}
        \arrow{s,t}{\alpha}
\node{ \MSpin^{S^1\fb}}
        \arrow{s,t}{\alpha^{S^1\fb}}
\\
\node{\Sigma \KO\vee \Sigma \KU}
        \arrow{e,t}{\tau^{\KO}}
\node{\KO}
        \arrow{e,t}{i^{\KO}}
\node{\KO^{S^1\fb}}
  \end{diagram}
\end{equation}
Here $\alpha \co \MSpin \to \KO$ and $\alpha^c\co \MSpin^c \to \KU$ are
the corresponding index maps.\footnote{For more about these maps, see
\cite{MR1166518}.} In particular, we have the index
homomorphism
$\alpha^{S^1\fb}\co \Omega^{\spin,S^1\fb}_n\to KO^{S^1\fb}_n$.
Here is the second main result, on the existence of
a conical psc-metric on a spin manifold with fibered
$S^1$-singularities:
\begin{thmB}
Let $M$ be a spin manifold with fibered $S^1$-singularities, of
dimension $n\ge 7$.  Assume that $\partial M$ is non-empty and connected,
and that $M$ and $\partial M$ are simply connected
{\lp}so that $\beta M$ is simply connected as well{\rp}. Then $M$ admits
a conical metric of \psc \ if and only if $\alpha^{S^1\fb}([M])$ vanishes in
$KO^{S^1\fb}_n$ under the above index map, and the relative index
for extending the lift to $\p M$ of some metric of {\psc} on $\beta M$
to $M$ vanishes.
\end{thmB}
\begin{proof}[Outline]
Just as in the $\bZ/k$-case, we show that the map $\alpha^{S^1\fb}$
provides an obstruction to the existence of a psc-metric and that the
existence of a psc-metric on $M$ depends only on the bordism class
$[M]\in \Omega^{\spin,S^1\fb}_n$. This guarantees the necessity in
Theorem B. Again we need more information on the spectra
$\KO^{S^1\fb}$ and $\MSpin^{S^1\fb}$. 
The final step consists of studying the kernels of the
index maps $\MSpin\to \ko$ and $\MSpin^c\to \ku$. In the case of
$\MSpin$, Stolz \cite{MR1189863,MR1259520} showed that 
$\ker\alpha\co\Omega_*^{\spin}\to ko_*$ is the image of a transfer map
$T_{\bullet}\co \Omega^{\spin}_*(BG) \to \Omega^{\spin}_{*+8}$,
where $G=\PSp(3)$ and the transfer amounts to taking the total
space of an $\bH\bP^2$-bundle. One can do something similar
in the (easier) case of $\ker\alpha^c\co\Omega_*^{\spin^c}\to ku_*$, where
this time (since $\MSpin^c$ splits $2$-locally as a sum of
$\Sigma^{4j}\ku$'s and some $\bZ/2$ Eilenberg-Mac Lane spectra;
see \cite[\S8]{MR0234475} and
\cite[p.\ 184]{MR1166518}), the transfer amounts to taking the total
space of $\bC\bP^2$-bundles instead of $\bH\bP^2$
bundles.\footnote{\ There is an additional complication one has to be
  careful about.  There are many spin$^c$-structures on $\bC\bP^2$,
  and we need one for which the image in $ku_4$ vanishes.  This is
  \emph{not} the one given by the complex structure, for which the
  image under $\alpha^c$ in $ku_4\cong \bZ$ is the Todd genus, which is $1$.}
Details of this argument may be found in Section \ref{sec:CP2}.
Then \eqref{eq:6} allows us to show that all elements of the kernel
$\ker \alpha^{S^1\fb} \subset \Omega^{\spin,S^1\fb}_n$
of the index map on which the secondary (relative) index also
vanishes are realized by
$\HP^2$- or $\CP^2$-bundles and hence can be represented by manifolds
with fibered $S^1$-singularities carrying a psc-metric.

The theorem also holds even when $\p M$ is disconnected, assuming
that each component of $\p M$ is simply connected and that the
appropriate indices vanish on each component of $\beta M$.
\end{proof}
\subsection{Plan of the paper}
The organization of the rest of the paper is quite straightforward.
Section \ref{sec:Zk} deals in detail with the case of fibered
$\bZ/k$-singularities, and includes all the details of the proof of
Theorem A. Section \ref{sec:S1} explains the precise definition
of conical psc-metrics in the case of fibered $S^1$-singularities,
and includes all the details of the proof of Theorem B, except
for two results of possibly independent interest which are needed
for the proof of sufficiency, namely
the spin$^c$ bordism theorem, which is proved in Section \ref{sec:surg},
and the theorem on $\bC\bP^2$-bundles, which is proved in Section
\ref{sec:CP2}.

\subsection{Acknowledgments}
We would like to thank Paolo Piazza for many useful suggestions
on the subject of this paper.  Paolo is a coauthor
on a subsequent paper \cite{BJP} dealing with other singularity types $P$.
We would also like to thank Bernd Ammann, Bernhard Hanke, and
Andr{\'e} Neves for organizing an excellent workshop at
the Mathematisches Forschungsinstitut Oberwolfach
in August 2017 on Analysis, Geometry and Topology of Positive Scalar
Curvature Metrics, which led to the present work.  We would also
like to thank the referees for helpful suggestions and criticism
regarding earlier versions of the paper.

\subsection{Note on this version}
\label{sec:note}
This is a corrected version of the paper published in
J.\ reine angew.\ Math.\ \textbf{803} (2023), 103--136 \cite{MR4649186}.
In the course of writing \cite{BJP1}, which depends heavily on certain
parts of this paper, we found two errors in the published version.
They are corrected here. Some of the calculations in the
corrected parts of Sections
\ref{sec:surg} and \ref{sec:CP2} of this version were assisted by
chats with the AI model Claude Fable 5.1 (Anthropic); the writing and
organization are entirely the work of the authors.

The changes are as follows:
\begin{enumerate}
\item The spin$^c$ bordism theorem, Theorem \ref{thm:bordismspinc},
  requires the additional hypothesis that
  $c_1(L)\co\allowbreak H_2(M)\to \bZ$ be surjective, i.e., that $c_1(L)$ be
  \emph{primitive}, or equivalently that the classifying map
  $M\to \BSpinc$ be a $2$-equivalence.  The proof has been corrected
  accordingly, and Lemma \ref{lem:obstr} and Propositions
  \ref{prop:counterexample1} and \ref{prop:counterexample2} give a
  counterexample to the statement without this hypothesis.
  Theorem \ref{thm:bordismspinc} is used elsewhere in the paper only in
  the proof of Theorem \ref{thm:S1psc} (the even-type case); as a
  consequence, in Theorems B, \ref{thm:S1bordism} and \ref{thm:S1psc}
  we now assume that $\p M$ (rather than just $\beta M$) is simply
  connected, which implies that $\beta M$ is simply connected as well.
\item The published proof at the prime $2$ asserted that the transfer
  map ``acts the same way on all the summands'' of the additive
  splitting of $\ko\wedge\Sigma^{-2}\bC\bP^\infty$ as a wedge of
  suspensions of $\ku$, which would reduce everything to the bottom
  summand.  This is unjustified, since the splitting is not one of
  $\ku$-module spectra, and in fact the reduction cannot work (Remark
  \ref{rem:bottomclasses}).  The $2$-local argument has therefore been
  replaced by a version of the bottom-cell construction of Hopkins and
  Hovey: the explicit $\bC\bP^2$-bundles $E_j\to\bH\bP^{j-1}$ of
  \eqref{eq:ej} detect the successive bottom classes modulo $2$
  (Lemmas \ref{lem:geometry}--\ref{lem:detect}), multiplying them by a
  section of the Atiyah--Bott--Shapiro orientation gives an
  equivalence (Proposition \ref{prop:Phi}), and one obtains the
  $2$-local identity $\ker\alpha^c=\im T_{\MSpin}+(e_1,e_2,\ldots)$ of
  \eqref{eq:2local}.  Corollary \ref{cor:nonspinpsc}, which is what is
  actually used in the proof of Theorem \ref{thm:S1psc}, is unchanged.
\end{enumerate}
The numbering of Theorems \ref{thm:codim3spinc}, \ref{thm:bordismspinc},
\ref{thm:transfer} and Corollary \ref{cor:nonspinpsc} agrees with that
of the published version.

\section{\texorpdfstring{The case of fibered $\bZ/k$-singularities}{The case of fibered Z/k-singularities}}
\label{sec:Zk}
This section will be about giving necessary and sufficient conditions
for a manifold with fibered $\bZ/k$-singularities to admit a
psc-metric, under the additional conditions that $M$ and $\partial M$
are connected, simply connected, and spin. In particular, we prove Theorem A.
\subsection{Some definitions}
\label{sec:Zkdef}
Recall that a compact \emph{closed manifold with fibered
  $\bZ/k$-singularities} means a compact smooth manifold $M$ with
boundary, equipped with with a smooth free action of
the group $P=\bZ/k$ on $\partial M$.
We denote by $\beta M$ the quotient space $\p M/{P}$;
we have a covering map $p\co \partial M\to \beta M$.
The \emph{associated singular space} is $M_\Sigma = M/\!\!\sim$,
where $\sim$ identifies any two points in $\partial M$
having the same image in $\beta M$. We use the notation $(M, \p M\!\to
\!\beta M)$ to emphasize the above $\bZ/k$-fibered structure.  Note
that a $\bZ/k$-manifold in the sense of Sullivan and Baas
\cite{MR0212811,MR0346824} is a special case.
\vspace{2mm}

Just as in the case of Sullivan-Baas singularities, we say that a
smooth manifold $\overline M$ (with non-empty regular boundary $\p
\overline M$) is a \emph{manifold with fibered $\bZ/k$-singularities
  with boundary} $\delta \overline M = M$ if the boundary $\p
\overline M$ is given a splitting $\p \overline M=\delta \overline
M\cup \p_1 \overline M$ together with free $\bZ/k$-action on $\p_1
\overline M$. It is required that $M \cap \p_1\overline M = \p M$ so
that the $\bZ/k$-action on $\p_1\overline M$ restricts to a given free
$\bZ/k$-action on $\p M$.  In particular, we obtain that the
$\bZ/k$-covering $\p_1\overline M\to \beta \overline M$ restricted to
$\p M$ coincides with the $\bZ/k$-covering $\p M \to \beta M$, where
$\p (\beta \overline M) = \beta M$.  We denote by $\overline
M_{\Sigma}$ the associated singular space where we identify those
points in $\p_1\overline M$ which belong to the same orbit under
$\bZ/k$-action. Note that then $\delta \overline M_{\Sigma}=
M_{\Sigma}$ by definition.
\begin{remark}
We should emphasize that $\overline M$ should be thought as a manifold
with corners, where its \emph{corner} is the manifold $ \p(\delta
\overline M)= \delta \overline M \cap \p_1\overline M = -
\p(\p_1\overline M)$;  see Fig.\ \ref{fig1}.
\end{remark}

\begin{figure}[!htb]
\begin{picture}(0,0)
\put(64,53){{\small $\overline M$}}
\put(54,127){{\small $M=\delta \overline M$}}
\put(264,80){{\small $\beta \overline M$}}
\put(154,80){{\small $\partial_1\overline M$}}
\put(194,75){{\small $\beta$}}
\put(160,70){\vector(1,0){70}}
\put(205,120){{\small $\beta M$}}
\put(90,102){{\small $\partial_1M$}}
\end{picture}
\includegraphics[height=1.7in]{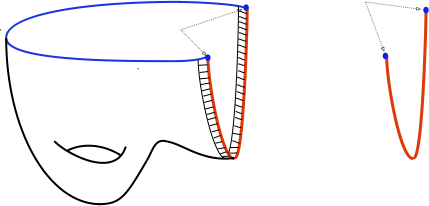}
\caption{A manifold $\overline M$ with fibered singularities and corners}
\label{fig1}
\end{figure}
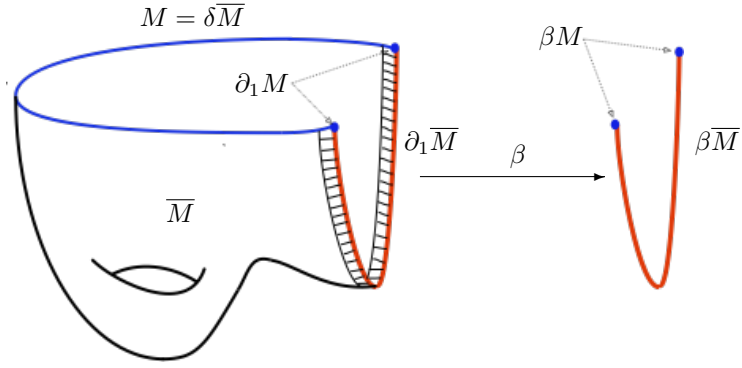 
A \emph{metric of {\psc}} on $\overline M$ is a Riemannian psc-metric
on $\overline M$, which is a product metric in a collar neighborhood
of the boundary $\partial \overline M$, and with the metric on
$\partial_1 \overline M$ $\bZ/k$-invariant. If $\delta \overline
M=\emptyset$, it defines a psc-metric on a closed manifold with
fibered $\bZ/k$-singularities.

\subsection{Bordism theory}
\label{sec:Zkbord}
Here we set up the bordism theory and the variant of $K$-theory that
will be needed in this section.  We use a slight modification of the
bordism theory of Baas \cite{MR0346824}. Below we assume that all
manifolds in this section are spin.
\vspace{2mm}

Let $X$ be a topological space. Then the group $\Omega^\spZk_n(X)$
consists of equivalence classes of maps $f \co M_\Sigma\to X$, where
$M_\Sigma$ is the singular space associated to an $n$-dimensional spin
manifold $(M,\partial M\!\to \!\beta M)$ with fibered
$\bZ/k$-singularities, with
$\bZ/k$ preserving the spin structure on $\p M$ (this is only an issue
if $k$ is even).  Two such maps $f \co M_\Sigma\to X$ and $f'\co
M'_\Sigma\to X$ are said to be equivalent if there exist a spin
manifold $\overline M$ with fibered $\bZ/k$-singularities with
boundary $\delta \overline M= M\sqcup -M'$ and a map $\overline f\co
\overline M \to X$ restricting to $f$ and $f'$ on $M$ and $M'$
respectively.
\vspace{2mm}

We use notation $\overline M \co  M \rightsquigarrow M'$.
In particular, the manifold $\beta \overline M$ gives a regular spin
bordism between closed manifolds: $\beta \overline M \co \beta M
\rightsquigarrow \beta M'$.
\vspace{2mm}

Exactly as in \cite{MR0346824}, it is easy to see that
$\Omega^\spZk_*(-)$ is a homology theory given by a spectrum
$\MSpin^{\bZ/k\fb}$. There is an obvious natural transformation $i \co
\Omega^{\spin}_*(-)\to \Omega^\spZk_*(-)$ which is given by
considering closed spin manifolds as manifolds with empty fibered
$\bZ/k$-singularity. The Bockstein operator $M \mapsto \beta M$
comes together with a map $h\co \beta M\to B\bZ/k$ classifying the
$\bZ/k$-covering $p\co  \p M \to \beta M$. This determines the
transformation
$\beta\co \Omega^\spZk_*(-)\to\Omega^{\spin}_*(-\times B\bZ/k)$
of degree $-1$. Then we have a transformation
\begin{equation*}
  \tau \co \Omega^{\spin}_*(-\times B\bZ/k)\to \Omega^{\spin}_*(-)
\end{equation*}
of degree $0$, given by a transfer: it sends a map $f\co N\to X \times
B\bZ/k$ to the lift $\tf\co \tN \to X$, where $\tN$ is the $k$-fold
cover of $N$ determined by the composite $\text{pr}_2\circ f \co N
\xrightarrow{f} X \times B\bZ/k \xrightarrow{\text{pr}_2} B\bZ/k$.
\vspace{2mm}

Just as in \cite[Theorem 3.2]{MR0346824} and \cite[\S2.1]{MR1857524},
there is an exact triangle of (unreduced) bordism theories
\begin{equation}
  \begin{diagram}
    \setlength{\dgARROWLENGTH}{1.95em}
  \node{\Omega^{\spin}_*(-)}
          \arrow[2]{e,t}{i}
  \node[2]{\Omega^\spZk_*(-)}
          \arrow{sw,t}{\beta}
  \\
  \node[2]{\Omega^{\spin}_*(-\times B\bZ/k)}
  \arrow{nw,t}{\tau}
  \end{diagram}
  \label{eq:bordtriZk}
\end{equation}
Restating the above assertion in slightly different language, we have
the following:
\begin{proposition}
  The transfer $\tau$ defines a map of bordism spectra $\tau\co \MSpin
  \wedge (B\bZ/k)_+ \to \MSpin$, and $\MSpin^{\bZ/k\fb}$ is the
  cofiber of this map.\footnote{\ Here, as usual, the subscript $_+$
    indicates the addition of a disjoint basepoint, and is needed to
    convert from reduced to unreduced homology theories.}
\label{prop:BZkcofib}
\end{proposition}
\begin{proof}
Clearly, this is just a restatement of the assertion above about the
exact triangle \eqref{eq:bordtriZk}.  To explain the geometry
involved, we write out the geometrical proof of the corresponding
exact sequence:
\begin{equation*}
\cdots \xrightarrow{\beta} \Omega^{\spin}_n(B\bZ/k) \xrightarrow{\tau}
\Omega^{\spin}_n \xrightarrow{i} \Omega^{\spZk}_n \xrightarrow{\beta}
\Omega^{\spin}_{n-1}(B\bZ/k) \xrightarrow{\tau} \cdots.
\end{equation*}
That $\beta\circ i=0$ is clear.  To see that $\tau\circ \beta = 0$,
observe that if $M$ has fibered $\bZ/k$-singularities and $\partial M
= \widetilde N$, $\beta M = N$, then $\tau\circ \beta([M]) =
[\widetilde N]$.  And $\widetilde N$ bounds as a spin manifold since
$\partial M = \widetilde N$.  To see that $i\circ \tau = 0$, note that
given $N\to B\bZ/k$ with corresponding covering $\widetilde N$, then
$i\circ \tau([N\to B\bZ/k]) = [\widetilde N]$. This is $0$ since
$\widetilde N$ bounds as a manifold with fibered
$\bZ/k$-singularities; take $W = \widetilde N\times [0, 1]$ with 
$\partial_1 W = \widetilde N\times \{0\}$, 
$\delta W = \widetilde N\times \{1\}$, $\beta W =
N\times \{0\}\cong N$. 

To get $\ker \beta \subseteq \image i$, observe that if $M$ has
fibered $\bZ/k$-singularities and $\beta M$ bounds as a spin manifold
with mapping to $B\bZ/k$, say $\beta M = \partial N$, then
the associated $k$-fold covering $\widetilde N$ has boundary $\partial M$,
and $M$ is bordant as a spin manifold with fibered $\bZ/k$-singularities
to the closed manifold $M\cup_{\partial M}-\widetilde N$.
To get $\ker\tau \subseteq \image \beta$, observe that if
$N\to B\bZ/k$ is a spin manifold with specified $k$-fold covering
$\widetilde N$, and if $\widetilde N$ is a spin boundary, say
$\partial M = \widetilde N$, then $M$ is a spin manifold with
fibered $\bZ/k$-singularities and with $\beta[M] = [N\to B\bZ/k]$.
Finally, to see that $\ker i \subseteq \image \tau$, suppose $M$ is a
closed manifold that bounds as a manifold with fibered
$\bZ/k$-singularities.  Then there is a $W$ with $\partial W$ decomposed
into two pieces: one of them a copy of $M$ and the other
$\widetilde M'$ projecting down to some $M'$; this shows $M$ is
bordant to $\widetilde M'$ in the image of $\tau$.
\end{proof}
\subsection{\texorpdfstring{Relevant $K$-theory}{Relevant K-theory}}
We denote by $\KO$ the spectrum representing real $K$-theory.  The
transfer map $\tau\co \MSpin \wedge (B\bZ/k)_+ \to \MSpin$ has its
analog in $K$-theories:
\begin{equation}\label{eq:transfer}
\tau^{\KO}\co \KO \wedge (B\bZ/k)_+ \to \KO,
\end{equation}
which is compatible with the transfer map for $k$-fold coverings of spin
manifolds.  Let $KO^{\Zk}_*(-)$ be the $K$-theory associated to the
cofiber $\KO^{\Zk}$ of the transfer map (\ref{eq:transfer}).

From this definition it follows  that the groups
$KO^{\Zk}_*$ fit into a long exact sequence
\begin{equation}
  \cdots \to KO_n(B\bZ/k) \xrightarrow{\tau^{\KO}} KO_n
  \xrightarrow{i^{\KO}} KO^{\Zk}_n \xrightarrow{\beta^{\KO}}
  KO_{n-1}(B\bZ/k) \xrightarrow{\tau^{\KO}} KO_{n-1} \to \cdots,
\label{eq:KOBzk}
\end{equation}
where when we decompose $KO_n(B\bZ/k)$ as $KO_n\oplus
\widetilde{KO}_n(B\bZ/k)$, $\tau^{\KO}$ is multiplication by $k$ on
the first summand and $0$ on the second.  A similar statement holds
for spin bordism.  Thus we have the following:
\begin{proposition}
 We have exact sequences
 \begin{equation} 
   0\to
    \Omega^{\spin}(\bZ/k)_n \to \Omega^{\spZk}_n \xrightarrow{\beta}
  \widetilde\Omega^{\spin}_{n-1}(B\bZ/k) \to 0
\label{eq:MspinBzk}
\end{equation}
and
\begin{equation}
  0 \to KO(\bZ/k)_n \to KO^{\Zk}_n \xrightarrow{\beta^{\KO}}
  \widetilde{KO}_{n-1}(B\bZ/k) \to 0.
\label{eq:KOBzk1}
\end{equation}
Here  $\Omega^{\spin}_*(\bZ/k)$ and
$KO_*(\bZ/k)$   denote   spin   bordism   and   real
$K$-theory with $\bZ/k$  coefficients.  If $k$ is odd,  then the group
$KO_n^{\Zk}$ vanishes except when $n$ is even, when $KO_n^{\Zk}$ is an
extension of $\widetilde{KO}_{n-1}(B\bZ/k)$  by $KO_n\otimes (\bZ/k)$,
and consists of odd torsion.
\label{prop:KOBzk}
\end{proposition}
\begin{proof}
  Most of this follows immediately from the exact sequences
  in Proposition \ref{prop:BZkcofib} and \eqref{eq:KOBzk}.
  
  When $k$ is odd, $KO_n(\pt;\bZ/k)= KO_n \otimes (\bZ/k)$,
  which
  is $\bZ/k$ when $4|n$, $0$ otherwise.  Since $\widetilde{H}_*(B\bZ/k)$
  is only nonzero in odd dimensions, the result follows.
\end{proof}
Finally, we need to consider the relationship between the spectra
$\MSpin^{\bZ/k\fb}$ and $\KO^{\Zk}$.  Let $\alpha\co\! \MSpin \to \KO$
be the usual Atiyah-Hitchin orientation for spin manifolds
(corresponding to the $KO$-index of the Dirac operator).  The
following result is a consequence of the naturality of the transfer:
\begin{proposition}
There is a map of spectra 
$\alpha^{\Zk}\co \MSpin^{\bZ/k\fb} \to \KO^{\Zk}$ making the following diagram
commute:
\begin{equation*}
\xymatrix{\MSpin \wedge (B\bZ/k)_+\ar^(.66)\tau[r] \ar^{\alpha\wedge 1}[d]
  &\MSpin \ar^{i}[r]\ar^\alpha[d]  &\MSpin^{\bZ/k\fb}\ar^{\alpha^{\Zk}}[d]\\
\KO \wedge (B\bZ/k)_+ \ar^(.66){\tau^{\KO}}[r] & \KO \ar^{i^{KO}}[r] &\KO^{\Zk}. }
\end{equation*}
\label{prop:alphaZk}
\end{proposition}
The last step is setting up the correct index map which will give the
obstruction to a psc-metric on a compact manifold with fibered
$\bZ/k$-singularities.  For this we need to recall the construction of
the assembly map (see for example \cite{MR1133900,MR1292018}). This is
much simpler in the case we need here of a finite group
$G$.\footnote{\ The \emph{Baum-Connes assembly map} (for a finite
  group) is a map $KO^G_*(\pt)\to
  KO_*(\bR[G])$, and is an isomorphism. The
  assembly map we use here is the composition of that map with the
  natural map $KO_*(BG) =  KO^G_*(EG) \to  KO^G_*(\pt)$.}
There is a map of spectra
\begin{equation*}
  \KO\wedge BG_+ \xrightarrow{\assemb} \KO(\bR[G]),
\end{equation*}
where $\KO(\bR[G])$ denotes the $K$-theory spectrum of the real group
ring, viewed as a Banach algebra, which can be defined in several
ways.  One method, developed by Loday \cite{MR0447373} in a slightly
different context, is to use the map on classifying spaces induced by
the inclusion $G\hookrightarrow GL(\bR[G])$.  An alternative is to
think of assembly as an index map.  Given a class in $KO_n(BG)$,
represented, say, by the Dirac operator $\Dirac_M$ on a closed spin
manifold $M^n$ with coefficients in some auxiliary bundle, together
with a map $M\to BG$ classifying a $G$-covering $\widetilde M$ of $M$,
\begin{equation*}
\assemb(\Dirac_M, M\to BG) = \ind^{\bR[G]}\left(\Dirac_{\widetilde M}\right),
\end{equation*}
where the right-hand side is the index of the lifted $G$-invariant
operator on $\widetilde M$ with coefficients in
$\bR[G]$. Since the assembly map for the
trivial group is just the identity map, we can peel this off and
consider also the reduced assembly map
\begin{equation*}
\widetilde{\assemb}\co
\KO\wedge BG \longrightarrow \KO(\widetilde{\bR[G]}),
\end{equation*}
where $\widetilde{\bR[G]}$ is the sum of the simple summands in the
group ring corresponding to all irreducible (real) representations
except for the trivial representation.  When $G$ is cyclic of odd
order, all representations except for the trivial representation are
of complex type, so the groups $KO_*(\widetilde{\bR[G]})$ are
torsion-free while $\widetilde{KO}_*(BG)$ is torsion.
Hence the map $\widetilde{\assemb}$ vanishes
{in homotopy}.  When $G$ is cyclic of even
order, it has exactly two irreducible representations of real type,
the trivial representation and the sign representation.  These are
homomorphisms $G\to O(1)=\{\pm 1\}$.  The remaining representations
are of complex type.  In this case, it is shown in \cite[Theorem
  2.5]{MR1133900} that $\widetilde{\assemb}$ is a split surjection
onto the torsion of $KO_*(\widetilde{\bR[G]})$,
which consists of $\bZ/2$ in each dimension
${*}\equiv 1,\,2\pmod 8$.
\begin{definition}
\label{def:Zkobstrmap}
The \emph{index obstruction map} (which appears in the statement
of Theorem A) is the map
$\ind^{\Zk}\co \Omega^\spZk_*\to KO^{\Zk}_*$ defined by
composing the map $\alpha^{\Zk}$ of Proposition \ref{prop:alphaZk}
with projection onto the inverse image of
the torsion in $\widetilde{KO}_{*-1}(\bR[\bZ/k])$. (Recall
that assembly gives a split surjection from
$\widetilde{KO}_{*-1}(B\bZ/k)$ onto the torsion in
$\widetilde{KO}_{*-1}(\bR[\bZ/k])$, and that we have the short exact
sequence \eqref{eq:KOBzk1}.)
\end{definition}
\subsection{Existence of a psc-metric}
The next step is to prove a ``bordism theorem'' for psc-metrics on
spin manifolds with fibered $\bZ/k$-singularities, which will give a
sufficient condition for such a manifold (under the conditions that
$M$ and $\partial M$ are connected and simply connected) to admit a
psc-metric.
\begin{theorem}[Cf.\ {\cite[Theorem 7.4(1)]{MR1857524}}]
\label{thm:Zkbordism}
Let $M$ be a spin manifold with non-empty fibered
$\bZ/k$-singularities, $\dim M =n \ge 6$. Assume that both $M$ and $\p M$
are connected and simply connected, and $\bZ/k$ preserves the spin
structure on $\p M$. If the bordism class
$[M]\in \Omega_n^\spZk$ contains a representative $M'$ admitting a
psc-metric, then $M$ also admits a psc-metric.\footnote{\ Note that
  the representative $M'\in [M]$ could be a manifold with empty
  singularity, but that $M$ has to have $\p M\neq \emptyset$.}
\end{theorem}
\begin{proof}
  The proof is essentially the same as that of
  {\cite[Theorem 7.4(1)]{MR1857524}}, using spin surgeries either
  in the interior of $M'$ or on $\beta M'$ (and then lifted
  to $\partial M'$).
\end{proof}
Consider the following commutative diagram:
\begin{equation}
\label{eq:pscbordismZk}
\xymatrix{
  \Omega^{\spinpsc}_n \ar^(.4)i[r] \ar@{^{(}->}[d] & \Omega^{\spinpsc, \,\bZ/k\fb}_n
  \ar^\beta[r] \ar@{^{(}->}[d] & 
\Omega^{\spinpsc}_{n-1} (B\bZ/k)\ar@{^{(}->}[d]\\
\Omega^{\spin}_n \ar^(.4)i[r] & \Omega^{\spin, \,\bZ/k\fb}_n \ar^\beta[r] & 
\Omega^{\spin}_{n-1} (B\bZ/k),}
\end{equation}
where the bottom row is exact by Proposition \ref{prop:alphaZk}, and the
top row consists of the subgroups of the groups on the bottom that are
generated by manifolds admitting psc-metrics.  It is easy to see that Theorem
\ref{thm:Zkbordism} implies the following
\begin{corollary}
\label{cor:pscbordismZk}
Let $M$ satisfy the condition of \textup{Theorem \ref{thm:Zkbordism}}.
Then it admits a psc-metric if and only if its bordism class $[M]\in
\Omega^{\spin, \,\bZ/k\fb}_n $ lies in the subgroup $\Omega^{\spinpsc,
  \,\bZ/k\fb}_n \subseteq \Omega^{\spin, \,\bZ/k\fb}_n $.
\end{corollary}  
We will now use the index obstruction map
$\ind^{\Zk}\co \Omega^\spZk_* \to KO^{\Zk}_*$
of Definition \ref{def:Zkobstrmap} in the main existence
theorem for psc-metric on manifolds with fibered
$\bZ/k$-singularities:
\begin{theorem}
\label{thm:Zkexistence}
Let $M$ be a simply connected spin manifold with fibered
$\bZ/k$-singularities, of dimension $n\ge 6$.  Assume that $\partial
M$ is non-empty, connected, and simply connected, and that the action
of $\bZ/k$ preserves the spin structure on $\p M$. If the element 
$\ind^{\Zk}([M]) \in KO_n^{\Zk}$
vanishes, then $M$ admits a psc-metric.
\end{theorem}
\begin{proof}
First suppose $k$ is odd. By Proposition \ref{prop:KOBzk},
$KO^{\Zk}_n$ vanishes when $n$ is odd, so in this case
$\ind^{\Zk}([M])$ automatically vanishes.  But in this case
$\Omega^{\spin}_n(\bZ/k)$ and $\widetilde{\Omega}^\spin_{n-1}(B\bZ/k)$
also vanish, so the result follows from Corollary \ref{cor:pscbordismZk}.

Now suppose $k$ is odd and $n$ is even.    Since the homology groups
$\widetilde{H}_*(B\bZ/k, \bZ)$ are concentrated in odd degrees, the
Atiyah-Hirzebruch spectral sequence for
$\widetilde{\Omega}^{\spin}_*(B\bZ/k)$ collapses and these
groups are also concentrated in
odd degrees, where all classes are linear combinations of generators
of the form $L^{2m+1}\times N^{4p}$, where $L^{2m+1}$ denotes a lens
space of dimension $\ge 3$. (There is a little trick here for getting
rid of generators of the form $S^1\times N^{4p}$ --- see
\cite[proof  of Theorem 1.3]{MR842428} and \cite{MR1484887}.)  But
$\bigl(D^{2m+2}\times N^{4p}, S^{2m+1}\times N^{4p}\to L^{2m+1}\times N^{4p}\bigr)$
is a spin manifold with fibered $\bZ/k$-singularities and positive
scalar curvature.  So subtracting off suitable {\psc} classes from
$[M]$, we can reduce to the case where $\beta([M])$ is trivial
in bordism, and thus (by the exact sequence \eqref{eq:MspinBzk})
that $[M]$ lies in the image of $\Omega^\spin(\bZ/k)_n$ under $i$.
Now recall that $\Omega^{\spin}_*$ is
a direct sum of $ko_*$, detected by the Atiyah-Hitchin obstruction
$\alpha$, and the image of the transfer map
\begin{equation*}
T\co\Omega^\spin_*(B\PSp(3))\to \Omega^\spin_{*+8},
\end{equation*}
representing total spaces of $\bH\bP^2$-bundles with {\psc}. Since
$\ind^{\Zk}([M])=0$ and $\ind^{\Zk}$ is faithful on
$ko_*\otimes (\bZ/k)$, the component of $[M]$ in $KO(\bZ/k)_n$ vanishes,
and so $[M]$ in $\Omega^{\spin}(\bZ/k)_n$ is
represented by a manifold of {\psc}.  This completes the proof when
$k$ is odd.

Now consider the case $k=2^r\cdot s$, where $r\ge 1$ and $s$ is odd.  The
situation with the odd torsion is exactly the same as before, so it is
no loss of generality to assume $k=2^r$.  Only the torsion behaves
differently than when $k$ is odd, so we can localize at $2$.  After
doing so, the spectrum $\MSpin$ splits (additively, not
multiplicatively) as a direct sum
\begin{equation*}
\MSpin_{(2)}= \ko \vee \mathsf{M},
\end{equation*}
where $\mathsf{M}$ consists of ``higher
summands'', i.e., suspensions of $\ko$ and $\ko\langle 2\rangle$ and
Eilenberg-Mac Lane spectra $\sH(\bZ/2)$, see \cite{MR0219077}.
Then as was shown by Stolz (see \cite{MR1189863,MR1259520}),
the spectrum $\mathsf{M}$ is in the image of the transfer map
\begin{equation*}
T\co\MSpin\wedge \Sigma^8 B\PSp(3)_+ \to \MSpin\,.
\end{equation*}
We proceed as before, assuming that $\ind^{\Zk}([M])=0$ and looking
at the image of $[M]$ in $\widetilde{\Omega}^{\spin}_{n-1}(B\bZ/k)$ in the
exact sequence \eqref{eq:MspinBzk}.  This is a sum of
$\widetilde{ko}_{n-1}(B\bZ/k)$ and a summand coming from
$\mathsf{M}$, represented by total spaces of $\bH\bP^2$-bundles
with {\psc}.  This second summand is the image under $\beta$
of $\bH\bP^2$-bundles over manifolds with fibered $\bZ/k$-singularities,
so these admit {\psc}.  Subtracting these off, we may assume that
the image of $[M]$ lies in $\widetilde{ko}_{n-1}(B\bZ/k)$.
The periodization map $ko_{n-1}(\bZ/k) \to KO_{n-1}(\bZ/k)$ will turn out to
be injective, so we can look at periodic $K$-homology instead.

Let's look at $\widetilde{KO}_{n-1} (B\bZ/2^r)$ in more detail.
By \cite[Theorem 2.5]{MR1133900}, this group is a direct sum of $2^{r-1}-1$
copies of the divisible group $\bZ/2^\infty = \bZ[\frac12]/\bZ$
for all even $n$, with one more copy of $\bZ/2^\infty$ when $n$ is
divisible by $4$, and a copy of $\bZ/2$ when
$n\equiv 2,\,3\pmod 8$.\footnote{\ All of this comes from dualizing the
  Atiyah-Segal Theorem and looking at the decomposition of the
  group ring $\bR[\bZ/2^r]\cong \bR^2\oplus \bC^{2^{r-1}-1}$.}
The copies of $\bZ/2$ map nontrivially under the Baum-Connes assembly map,
and lie in the image of corresponding summands in 
$\widetilde{ko}_{n-1}(B\bZ/2^r)$, by \cite[Lemma 2.8]{MR1484887}.
The analysis of the Atiyah-Hirzebruch spectral sequence
in the proof of that lemma shows in fact that the ``periodization'' map
\begin{equation*}
ko_{n-1} (B\bZ/2^r)\to KO_{n-1} (B\bZ/2^r)
\end{equation*}
is injective in all degrees.
So if $[M]$ is in the kernel of $\ind^{\bZ/2^r}$
we can assume $M$ represents a bordism class corresponding to
one of the $\bZ/2^\infty$ summands in $\widetilde{KO}_{n-1}(B\bZ/2^r)$.
By \cite[\S5]{MR1484887}, that means we can assume $\beta M$
is represented by a linear combination of lens spaces (when
$n\equiv 0\pmod4$) or lens spaces bundles over $S^2$ (when
$n\equiv 2\pmod4$).

First suppose $\beta M$ is represented by a lens space.  This case
is very much like the corresponding step in the case with $k$ odd.
We can take a disk $D^n$, equipped
with the constant-curvature metric from
the upper hemisphere in $S^n$, and $\beta M$ can then be identified
with the quotient of $\partial D^n = S^{n-1}$ by an isometric
$\bZ/k$-action.  So subtracting this off from the bordism class of $M$,
we can reduce to the case where $M$ is a spin boundary and admits
{\psc}.  This covers the case $n\equiv 0\pmod4$.  When $n\equiv 2\pmod4$,
the situation is only slightly more complicated.  Because of
\cite{MR1484887}, we proceed as in the case $n\equiv 0\pmod4$ but with
disks replaced by disk bundles over $S^2$.  These still have {\psc}
so again we can subtract them off and assume the class $[M]$
lies in the image of $\Omega^\spin(\bZ/2^r)_n$.  Again this group
splits as a direct sum of $ko(\bZ/2^r)_n$ and a piece coming
from the spectrum $\mathsf{M}$, represented by total spaces
of $\bH\bP^2$-bundles with {\psc}.  This latter piece is no problem,
so finally we can assume $[M]$ is in the image of $ko(\bZ/2^r)_n$
and has vanishing index obstruction.  As before that means it
represents a bordism class corresponding to
one of the $\bZ/2^\infty$ summands in $\widetilde{KO}_n(B\bZ/2^r)$,
which come from lens spaces or lens spaces bundles over $S^2$.  These
have {\psc} so we are done.
\end{proof}
\subsection{Obstruction to the existence of a psc-metric} 
\label{sec:Zknec}
Let $M$ be a manifold with fibered $\bZ/k$-singularities as above.
Then the locally compact groupoid $C^*_{\bR}(M;\bZ/k)$ is defined
exactly as in \cite[\S2]{MR2011114}; it has unit space $N$, where
$N=M\cup_{\partial M}{\partial M}\times [0, \infty)$ is $M$ with
  infinite cylindrical ends added to the boundary, and the groupoid
  structure is defined by the equivalence relation $\sim$ which is
  trivial on $M$ itself and given by the action of $P = \bZ/k$
  on ${\partial M}\times (0, \infty)$. Recall that $C^*_{\bR}(M;\bZ/k)$
  is an extension
\begin{equation}
0 \to \Gamma_0((0,\infty)\times \beta M, \cA) \to
C^*_{\bR}(M;\bZ/k) \to C^\bR(M) \to 0,
\label{eq:C*MZk}
\end{equation}
where $\cA$ is a bundle over $(0,\infty)\times \beta M$ with
fibers isomorphic to $M_k(\bR)$ and with trivial Dixmier-Douady class,
and there is also a ``target'' real {\Ca} $C^*_{\bR}(\pt;\bZ/k)$
which is an extension
\begin{equation}
0 \to C_0^\bR((0,\infty), M_k(\bR)) \to
C^*_{\bR}(\pt;\bZ/k) \to \bR \to 0.
\label{eq:C*ptZk}
\end{equation}
The long exact sequence of the extension \eqref{eq:C*ptZk} in 
$KO$-homology\footnote{We're following a common abuse of
terminology and calling $K$-homology or $KO$-homology
the theory that goes by this name on spaces, even though it
is contravariant on {\Ca}s.}  has connecting map
\begin{multline*}
KO_{n-1}(\pt)\cong KO_n((0,\infty))\cong
KO^{-n}(C_0^\bR((0,\infty), M_k(\bR))) 
\to
KO^{1-n}(\bR)\cong KO_{n-1}(\pt)
\end{multline*}
induced by the unital ring map $\bR\to M_k(\bR)$, and hence given by
multiplication by $k$.  Thus $KO^{-n}(C^*_{\bR}(\pt;\bZ/k))\cong
KO_n(\pt;\bZ/k)$.  Here are our main results about this situation.
\begin{theorem}[Obstruction Theorem]
Let $M$ be a spin manifold with fibered $\bZ/k$-singularities.  Assume
that the action of $\bZ/k$ on $\partial M$ preserves the induced spin
structure. Then the Dirac operator on $M$ has a well-defined index 
\begin{equation*}
  \ind^{\bZ/k}(M) \in KO^{-n}(C^*_\bR(\pt;\bZ/k))\cong
  KO_{n}(\pt;\bZ/k)
\end{equation*}
which is independent of the choice of Riemannian
metric on $M$.  If $M$ has psc-metric, then this index must vanish.
\label{thm:Zkindex} 
\end{theorem}
\begin{proof}
This is essentially \cite[Definition 4.2]{MR2011114} together with
\cite[Theorem 5.3]{MR2011114}, except for the fact
that the covering map $p\co \partial M\to \beta M$ need not split.
If one looks at the proof given there, the splitting is never used,
so this is not a problem.
\end{proof}
\subsection{Proof of Theorem A}
Now we can put all the pieces together.  Recall the statement:
\begin{theorem}[Theorem A]
Let $M$ be a spin manifold with fibered $\bZ/k$-singularities, $\dim M
= n\geq 6$. Assume that $M$ and $\partial M$ are both non-empty and
simply connected, and that $\bZ/k$ preserves the spin structure
on $\p M$.  Then the vanishing of the
index obstruction $\ind^{\bZ/k}(M)$ is necessary
and sufficient for $M$ to admit a psc-metric.
\label{thm:thmA} 
\end{theorem}
\begin{proof}
Sufficiency is in Theorem \ref{thm:Zkexistence}.
For necessity, there are two pieces, due to the fact that  
$\ind^{\bZ/k}$ has two components.  The necessity of vanishing
of the first component, in $KO_n(\pt;\bZ/k)$,
is covered by Theorem \ref{thm:Zkindex}.  The necessity of
vanishing of the second component, in $\widetilde{\KO}(\bR[\bZ/k])$,
is simply \cite[Theorem 3.3]{MR842428}
or \cite[Theorem 3.1]{MR1133900}, applied to $\beta M$.
\end{proof}
\begin{remark}
  It is easy to modify the statement and proof to apply to the situation
  where $\p M$ is disconnected, but each component of $\p M$ is
  simply connected.  We leave the details to the reader.
\end{remark}  
\section{\texorpdfstring{The case $P=S^1$}{The case P=S1}}
\label{sec:S1}
\subsection{The goals}
\label{subsec:goals}
In this section we modify the same program for
manifolds with $S^1$-fibered singularities.  Since the cone on $S^1$
is $\bR^2$, in this case the pseudomanifold $M_\Sigma$ is actually a
smooth manifold, but with a distinguished codimension-$2$ submanifold
$\beta M$. Our aim is prove Theorem B, thereby answering a version of
the question in Section \ref{q:main}.  Along the way we will prove
some partial results about cylindrical metrics of {\psc}.
\subsection{The setting}
\label{subsec:Def-Ex}
The following is the exact analogue of the definitions
in Section \ref{sec:Zkdef}, but we want to restrict the metric
to have a particular form near $\beta M$. 
\begin{definition}
\label{def:S1fib}
A compact \emph{manifold with $S^1$-fibered
  singularities} will mean a compact smooth manifold $M$ with
boundary, equipped with a smooth free action of $P = S^1$ on
$\partial M$. We denote by $\beta M$ the quotient space $\partial M/S^1$.
We have a principal $S^1$-bundle $p\co \partial M \to \beta M$
which we identify with the unit circle bundle of
a corresponding complex line bundle $\widetilde p\co L \to \beta M$.
The \emph{associated manifold}
is $M_\Sigma = M/\!\!\sim$, where $\sim$ identifies any two points in
$\partial M$ lying in the same $S^1$-orbit.
\end{definition}
We notice that $M_\Sigma$ is indeed a manifold since the
cone on $S^1$ is $\bR^2$; in fact, we can identify $M_\Sigma$ with the
smooth manifold obtained by gluing the disk bundle $T$ of $L$ to $M$
along their common boundary $\p M\cong \p T$. (Of course we do this in
such a way that the resulting manifold is oriented, so that one should
think of $\p T = -\p M$, and this, rather than $\p M$, is really the
disk bundle of $L$.)
\begin{definition}
  \label{def:cyl-metric}
A \emph{cylindrical} \emph{metric of {\psc}}
on $M$ is a Riemannian metric of {\psc} on $M$, which is
a product metric in a collar neighborhood of the boundary
$\partial M$, and with the metric $S^1$-invariant on $\partial M$.
\end{definition}
Clearly the disk bundle $T$ also serves as a
tubular neighborhood of $\beta M\subset M_\Sigma$, and comes with
the projection map $\tilde p\co T \to \beta M$.
\begin{definition}
\label{def:conicalmetric}  
A \emph{conical}
metric of {\psc} on $M$ is a metric of {\psc} on $M_\Sigma$ which is
$S^1$-invariant on the tubular neighborhood $T$ of $\beta M$, and
which on this tubular neighborhood has the form $(\widetilde
p)^*g_{\beta M} + h$, where $g_{\beta M}$ is a metric of {\psc} on
${\beta M}$ and $h$ is a hermitian metric on the line bundle $L$
over ${\beta M}$, smoothed out near the boundary
$\partial T = - \partial M$
to match a metric of the form $dr^2 + g_{\p M}$ on $M$.
More precisely, we fix a connection on the line bundle $L$, giving
a choice of horizontal tangent space at each point in the total
space of $L$, and make the horizontal and vertical spaces perpendicular,
with the metric $(\widetilde p)^*g_{\beta M}$
on the horizontal tangent space and the metric $h$ on the
vertical tangent space. In a neighborhood of $\p M$,
with $r$ the distance to $\beta M$, the metric will locally look
like $dr^2 + f(r)\,g_{\p M}$, where $f$ is $C^\infty$,
$f(r)\equiv 1$ for $r\ge 1$, and $f(r)$
smoothly transitions to $r^2$ for $r<1-\varepsilon$ for some
small $\varepsilon$.

\end{definition}
This definition is consistent with the idea that a conical metric
should have the approximate form
$dr^2 + r^2g_{S^1} + (\widetilde p)^*g_{\beta M}$ near ${\beta M}$.

We notice that since the metric on a neighborhood of
$\p M$ is either exactly (in the cylindrical case) or approximately
(in the conical case) of the form $g_{\p M} + dr^2$, if $M$ admits a
metric of {\psc} (in either sense), so does $\p M$.  By \cite[Theorem
  C]{BB}, this implies that $\beta M$ has a metric of {\psc}.  \qed

\begin{remark}
\label{rem:BB}  
Since the paper \cite{BB} is published in a rather inaccessible place
(not even indexed by \emph{MathSciNet} or \emph{Zentralblatt}),
for the benefit of the reader, we repeat the essence of the argument
here.  Theorem C from \cite{BB} states that if one
has a closed $n$-manifold $N$ with a free action of $S^1$, giving a
principal $S^1$-bundle $p\co N\to B$, then the
manifold $N$ admits an $S^1$-invariant metric of {\psc} if and only
if $B$ admits {\psc}.  The ``if'' direction is relatively easy given
\cite[Theorem 3.5]{MR0262984}; a metric $g_B$ of {\psc} on $B$ lifts
to an $S^1$-invariant metric $g_N$ on $N$ with totally geodesic
fibers. Let $\kappa_N$ and $\kappa_B$ be the scalar curvatures of $N$
and $B$, respectively.  By the O'Neill formulas \cite{MR0200865},
$\kappa_B = \kappa_N + \Vert A\Vert^2$, where $A$ is the O'Neill
tensor, and rescaling $g_N$ on the circle fibers, one can make $\Vert
A\Vert^2$ small so that $\kappa_N>0$.

For the other direction, assume we have an $S^1$-invariant metric
$g_N$ of {\psc} on $N$, and let $g_B$ be the induced metric on $B$ and
let $\kappa_B$ and $\Delta_B$ be its scalar curvature and
Laplacian. (We use the analysts' sign convention in which $\Delta_B$
has nonpositive spectrum.) We need to have $n\ge 3$, since if $N$ is a
closed $2$-manifold fibering over $S^1$, then $N$ is a torus or Klein
bottle and does not admit {\psc}.  Let $f$ be the function
on $B$ which at $x\in B$ gives the length of the $S^1$-fiber over $x$.
B{\'e}rard-Bergery computes from the O'Neill formulas that if
$\kappa_N>0$, then
\begin{equation*}
  \kappa_B + 2\frac{\Delta_B f}{f} > 0.
\end{equation*}
From this it follows that after making the conformal change
$\widetilde g_B = f^{\frac{4}{n-2}} g_B$
the manifold $(B,\widetilde g_B)$ has  {\psc}. \qed
\end{remark}
\begin{example}
\label{ex:K3}
(i) This example is adapted from \cite[Example 9.2]{BB}.  Let $\beta
M$ be a K3 surface (which is a simply connected spin $4$-manifold with
nonzero $\widehat A$-genus). Then $\beta M$ does not admit a metric of
{\psc}, but there is a circle bundle $\p M$ over $\beta M$ with simply
connected total space $\p M$.  To construct such a bundle, it is
enough to choose a primitive element of
\[
H^2(\beta M,\bZ)\cong \bZ^{22}
\]
for the first Chern class $c_1$.  The manifold $\p M$ is
necessarily spin, since $T_{\p M}\cong p^*T_{\beta M} \oplus V$, where
$V$ is the real tangent line bundle along the circle fibers, which is
trivial, and thus $w_2(T_{\p M}) = p^*w_2(T_{\beta M}) = 0$. $\p M$ is
a spin boundary, since $\Omega^\spin_5=0$.  So choosing $M$ to have
boundary $\p M$, we get a spin $6$-manifold $M$ with $S^1$-fibered
singularities, which we can choose to be simply connected.  By Van
Kampen's Theorem, the associated manifold $M_\Sigma$
is also simply connected.  As will be seen later (cf.\ Remark
\ref{rem:spinstr}), $M_\Sigma$ is spin$^c$ but not spin, and it admits
a metric of {\psc} by \cite[Theorem C]{MR577131}, but $\beta M$ does
not. This gives a nontrivial example where the singular manifold
$M_\Sigma$ admits a metric of {\psc} but $M$ itself does not admit
{\psc} (either conical or cylindrical) in the sense of Definition
\ref{def:S1fib}.
\vspace{1mm}

(ii) We can also construct a ``complementary'' example where $\beta M$
admits a metric of {\psc}, and $\p M$ admits an $S^1$-invariant metric
of {\psc}, but $M$ does not admit {\psc}
(either conical or cylindrical) in the sense of Definition
\ref{def:S1fib}.  Take $\Sigma^{10}$ to be an exotic $10$-sphere with
$\alpha(\Sigma)$ nontrivial in $KO_{10}\cong \bZ/2$. Cut out a disk and
let $M = \Sigma \smallsetminus \mathring{D}^{10}$.  Then
$\p M=S^9$, which of course has a free $S^1$ action with quotient
space $\beta M = \bC\bP^4$.  So one can glue in a standard
tubular neighborhood of $\bC\bP^4$ to get $M_\Sigma$ with the
homotopy type of $\bC\bP^5$.  Note that $M_\Sigma$, $M$, $\p M$, and
$\beta M$ are all simply connected, and that $\p M$ and
$\beta M$ admit standard metrics of positive curvature, but that this
metric on $\p M$ cannot extend to a metric of positive scalar
curvature on $M$, since if it did, patching with a standard metric
on $T = D^{10}$ would give a metric of {\psc} on $\Sigma$, which is
ruled out by the $\alpha$-invariant.
\vspace{1mm}

(iii) The following example is one where $M_\Sigma$ admits {\psc} but
there is a torsion obstruction to {\psc} in the sense of Definition
\ref{def:S1fib}, due to failure of $\beta M$ to admit a metric of
    {\psc}.  Start with $\beta M = \Sigma^{10}\# \bC\bP^5$, where
    $\Sigma^{10}$ is a homotopy $10$-sphere with nonzero
    $\alpha$-invariant (i.e., representing the generator of
    $KO_{10}$).  Since the $\alpha$-invariant is additive on connected
    sums, $\beta M$ does not admit a metric of {\psc}.  However, the
    manifold $\beta M$ is a fake complex projective space, so it
    admits a principal $S^1$-bundle for which the total space $\p M$
    is a homotopy $11$-sphere.  There being no torsion in
    $\Omega^\spin_{11}$, the exotic sphere $\p M$ is a spin boundary
    and we can choose a spin $12$-manifold $M$ with $S^1$-fibered
    singularities having boundary this circle bundle over $\beta
    M$. By \cite[Theorem C]{BB}, there is no $S^1$-invariant metric of
    {\psc} on $\p M$.  \qed
\end{example}
\begin{remark}
\label{rem:spinstr}
Let $M$ be a spin manifold with $S^1$-fibered singularities as in
Definition \ref{def:S1fib}. Assume for simplicity that $\partial M$
is connected.  Since $M$ is a compact spin manifold
with boundary, $\partial M$ carries a natural spin structure.
The principal bundle $p\co \partial M\to \beta M$ is the circle
bundle associated to a complex line bundle
$\widetilde p\co L \to \beta M$.  The tangent bundle of $\partial M$
splits as the direct sum of $p^*(T(\beta M))$ and the real line
bundle along the circle fibers, which is trivial.  Thus
$0 = w_2(\partial M) = p^*(w_2(\beta M))$.  So $w_2(\beta M)\in \ker p^*$,
which by the Gysin sequence is the $\bF_2$-span of $c_1(L)$ reduced
mod $2$.  Hence either $\beta M$ is spin, or else $L$ is nontrivial
and $w_2(\beta M) = c_1(L)\!\!\mod 2$.

Note that the tangent bundle of the tubular neighborhood $N$ of
$\beta M$ coincides with the direct sum
$\widetilde p^*T(\beta M)\oplus \widetilde p^*L$, and that $N$ has a
deformation retraction down to $\beta M$. By the additivity formula
for Stiefel-Whitney classes, plus the fact that for a complex line
bundle viewed as a real $2$-plane bundle,
$c_1$ reduces mod $2$ to $w_2$, we see that $N$ (and hence also $M_\Sigma$) is
spin exactly when $w_2(\beta M)= c_1(L)\!\!\mod 2$.  Furthermore, if
$\p M$ and $\beta M$ are both simply connected, then from the long
exact homotopy sequence
\[
0\to \pi_2(\p M) \xrightarrow{p_*} \pi_2(\beta M) \xrightarrow{\p}
\pi_1(S^1) = \bZ\to 0,
\]
we see that $c_1(L)\in H^2(\beta M; \bZ) \cong \Hom(\pi_2(\beta M), \bZ)$
has to be non-zero mod $2$, so $\beta M$ and $M_\Sigma$ cannot both
be spin.  The example given before of $\beta M = \bC\bP^{n-1}$ and
$M_\Sigma = \bC\bP^{n}$ is instructive in this regard.

The fact that $\beta M$ and $M_\Sigma$ cannot both be spin in this
case, whereas they will both be spin when the link
$P$ is higher-dimensional,
can be explained by the fact that a spin structure on a manifold
is equivalent to a trivialization of the tangent bundle over the
$2$-skeleton (for a CW decomposition) (see, e.g.,
\cite[Theorem II.2.10]{lawson89:_spin}). Since, in our case,
$\beta M$ has codimension $2$, we do not have room to push
$\beta M$ away from this $2$-skeleton.  For that we would need
the singular stratum to have codimension at least $3$.
\end{remark}
\begin{example}
\label{ex:spinc}
The following is a more general version of the same example.  Let
$M_\Sigma$ be a simply connected spin$^c$ manifold which is
\emph{not} spin, i.e., with $w_2(M)\ne 0$. ($\bC\bP^n$ is an example
if $n$ is even.) The spin$^c$
condition means we have fixed a complex line bundle $\cL$
on $M_\Sigma$ such that
$c_1(\cL)$ reduces mod $2$ to $w_2(M_\Sigma)$.  Choose a submanifold
$\beta M$ (which will have codimension $2$ in $M_\Sigma$) dual to $\cL$;
by construction, $c_1(\cL)$, and thus also $w_2(M_\Sigma)$, is trivial on the
complement $M$ of a tubular neighborhood $N$ of $\beta M$,
so we get a spin manifold $M$ with boundary $\partial M$,
and $\partial M$ is a circle bundle over $\beta M$. We get the line bundle
$L$ associated to this circle bundle by pulling back $\cL$ via the
inclusion $\iota\co\beta M\hookrightarrow M_\Sigma$, and $N$ can be
identified with the disk bundle of $L$.  Furthermore,
$\beta M$ is a spin manifold, since
\[
w_2(\beta M) + \bigl(c_1(L)\!\!\!\mod 2\bigr) =
w_2\left(\left.TM_\Sigma\right\vert_{\beta M}\right) =
\iota^* w_2(M_\Sigma) = \bigl(c_1(L)\!\!\!\mod 2\bigr),
\]
which says that $w_2(\beta M)=0$. \qed
\end{example}
\subsection{\texorpdfstring{The case of $P=S^1$: geometry of the tubular neighborhood}{The case of P=S1: geometry of the tubular neighborhood}}
\label{sec:S1tube}
For use later on, we want to study in more detail the geometry of the
tubular neighborhood $T$ of $\beta M$ in a conical metric.  Just for
this subsection, to simplify notation, we replace $\beta M$ by $B$ and
$\widetilde p$ by $p$. Consider the following general setting:

Let $B$ be a Riemannian manifold, and let $p\co L\to B$
be a complex line bundle over $B$, equipped with
a connection $\nabla^L$ and a hermitian metric $h$.  View the total
space of $L$ as a Riemannian manifold with the associated conical
metric as in \textup{Definition \ref{def:conicalmetric}}. Then
the projection $p$ becomes a Riemannian submersion
with totally geodesic flat fibers. We denote by $A$
and $T$ the associated O'Neill
tensors \textup{(see \cite{MR0200865})}, and by $F$
the curvature $2$-form of the line bundle $L$. \footnote{We
normalize $F$ to be purely imaginary; many authors normalize it to
be real. This would introduce a factor of $i$ but wouldn't change
the formula for the sectional curvature.}
\begin{proposition}
\label{prop:adaptedmetric}
Let $p\co L\to B$ be a complex line bundle over $B$ with connection
$\nabla^L$ and hermitian metric $h$, as above.  Denote by $X$ and $Y$
horizontal vector fields projecting to vector fields $X_*$ and $Y_*$
on $B$, and by $V$ and $W$ vertical vector fields. Then
\begin{enumerate}
\item[{\rm (i)}] $A_XV=0$ and $A_XY
  = \frac{1}{2i}F(X_*,Y_*)\partial_z$, with $\partial_z$ denoting
  differentiation by $z$ in the fiber direction;
\item[{\rm (ii)}] for $v, \,w$ vertical unit tangent vectors and $x,
  \,y$ horizontal unit tangent vectors, the sectional curvatures of
  $L$ are 
$$
K_{vw}=K_{vx}=0, \ \ \ K_{xy}=\widehat K_{x_*y_*} - \frac{3}{4}\vert
F(x_*,y_*)\vert^2\,\langle\partial_z,\partial_z\rangle_h;
$$
here $x_*=dp(x)$, $y_*=dp(y)$, and $\widehat K$
is the sectional curvature on the base manifold $B$.
\end{enumerate}
\end{proposition}
\begin{proof}
By the construction of the conical metric, the
derivative $dp$ is an isometry on horizontal
vectors, and so $p$ is a Riemannian submersion.  Similarly, the metric
on the fibers is just the flat Euclidean metric defined by $h$, and
the fibers are totally geodesic.  So the O'Neill tensor $T$ vanishes
identically. We have $A_XV=0$ since if we lift a geodesic in $B$ to a
horizontal geodesic in $L$, then locally, this geodesic and a vertical
straight line span a totally geodesic flat submanifold in $L$.  Thus
everything follows from \cite[Corollary 1, p.\ 465]{MR0200865} once we
check that $A_XY = \frac{1}{2i}F(X_*,Y_*)\partial_z$.  This in turn
follows from the definition of the covariant derivative and curvature
in terms of the connection.  Let $X_*$ and $Y_*$ partial derivatives
with respect to local coordinates in $B$, lifted up to basic vector
fields on $L$.  Then $[X_*,Y_*]=0$ on $B$ and
\[
A_XY = \frac12 [X, Y]^{\mathcal V} = \frac{\partial_z}{2i}(p^*F)(X,Y)
= \frac{1}{2i}F(X_*,Y_*){\partial_z}
\]
(see for example \cite[Ch.\ 1]{Charles} for some of the relevant
calculations).  
\end{proof}
\begin{corollary}
\label{cor:pscconical}  
Let $B$ be a closed Riemannian manifold, and let $p\co L\to B$ be a
complex line bundle over $M$, equipped with a connection and a
hermitian metric $h$.  View the total space of $L$
as a Riemannian manifold with the associated conical metric.
If $L$ has {\psc}, then so does $B$.  Conversely, if $B$ has {\psc},
then by rescaling the metric $h$ we can arrange for the scalar
curvature of $L$ to be bounded below by a positive constant.
\end{corollary}
\begin{proof}
Immediate from the sectional curvature formulas for $L$ in
Proposition \ref{prop:adaptedmetric}.
\end{proof}
\subsection{\texorpdfstring{The case of $P=S^1$: obstruction theory}{The case of P=S1: obstruction theory}}
\label{sec:S1obstr}
Now we want to set up an obstruction theory that will show that in
some cases, a spin manifold $M$ with $S^1$-fibered singularities
(in the sense of Definition \ref{def:S1fib})
does not carry a metric of {\psc}.  For our first steps
it will not actually matter whether we use the
cylindrical metric definition (Definition \ref{def:cyl-metric})
or conical metric definition (Definition \ref{def:conicalmetric}),
but we concentrate on the latter, 
which is in many respects more natural.  We
assume that $M$ and $\beta M$ are both simply connected.
Then $M_\Sigma$ will be as well (by Van Kampen's Theorem). As
we saw right after Definition \ref{def:conicalmetric},
there is an obvious necessary condition: $\beta M$
must admit a metric of {\psc} (if $M$ carries either a cylindrical
or a conical psc-metric), and in addition, in the conical case 
$M_\Sigma$ must admit a psc-metric (in the usual sense for closed manifolds).

Let $M$ be a simply connected spin $n$-manifold with $S^1$-fibered
singularities and $\p M\to \beta M$ be the corresponding
$S^1$-bundle.  Assume that $\beta M$ is also simply
connected. There are two cases to consider: when $\beta M$ is spin or
not spin.

If $\beta M$ is spin, we assume that $\alpha([\beta M])=0$ in $KO_{n-2}$,
where $\alpha\co \Omega_{*}^{\spin}\to KO_{*}$ is the index map. Then
$\beta M$ admits a metric of {\psc} by \textup{\cite{MR1189863}}.
If $\beta M$ is not spin,
then $\beta M$ admits a metric of {\psc} by
\textup{\cite{MR577131}}.

In both cases, we fix a metric $g_{\beta M}$ of {\psc} on $\beta
M$. As pointed out in the proof of \textup{\cite[Theorem
    C]{MR577131}}, once one has normalized the arclength of the $S^1$
fibers of the map $p\co \p M\to \beta M$, this determines uniquely via
\textup{\cite[Theorem 3.5]{MR0262984}} an $S^1$-invariant metric
$g_{\p M}$ on $\p M$ with totally geodesic fibers, which will have
{\psc} if we make the fibers short enough.  Attach an infinite
cylinder $\p M \times [0, \infty)$ to $M$ along $\p M$, give
the end the product metric $g_{\p M} + dr^2$, and extend to a
metric $g$ on $M\cup_{\p M} (\p M \times [0, \infty))$.
Since $g$ has {\psc} except perhaps on a compact set
{\lp}namely, $M${\rp}, the $\text{Cl}_n$-linear Dirac
operator on this open manifold is Fredholm and has a
well-defined index in $KO_n$.  We call it the
\emph{relative index} since it is essentially the same as
the relative index invariant defined in
\cite[\S4]{MR720933}.
\begin{proposition}
\label{prop:relat-ind}
Under these circumstances, the relative index in $KO_n$ is an
obstruction to $g_{\p M}$ extending to a metric of {\psc} on $M$ in
either of the senses of \textup{Definition \ref{def:S1fib}}.
\end{proposition}
\begin{proof}
This is obvious, as existence of a metric of {\psc} on $M$ extending
the given $S^1$-invariant metric $g_{\p M}$ on $\p M$ implies that
the Dirac operator has spectrum bounded away from $0$, by the Lichnerowicz
identity, and thus has vanishing index.  Note by the way that
the relative index can depend on the isotopy class of the {\psc}
metric on $\beta M$, which was used to define the
$S^1$-invariant metric $g_{\p M}$ on $\p M$.
\end{proof}
\begin{remark}
  While computing the relative index of Proposition \ref{prop:relat-ind}
  is not always so easy, there is one case where one knows that it
  vanishes.  Namely, if the free $S^1$-action on $\p M$ extends to
  a (non-free) action on $M$ with some component of the fixed set
  $M^{S^1}$ having codimension $2$, then $M$ admits an $S^1$-invariant
  psc-metric by \cite[Theorem 2.4]{MR3449263}.  (The theorem
  is stated for closed manifolds but the same argument applies to
  our case.)
\end{remark}  

Next we state and prove another obstruction theorem which appears
in the ``only if'' part of Theorem B.  This result is an
adaption of \cite[pp.\ 715--716]{MR1857524}, which corresponds to
the special case where $p$ is a trivial bundle and
$\p M = \eta\times \beta M$, where $\eta$ is $S^1$ with its
non-bounding spin structure and $\beta M$ is spin.
\begin{theorem}
\label{thm:complexDiracobstr}
Let $(M, \p M)$ be a spin $n$-manifold with $S^1$-fibered singularities,
determined by an $S^1$-bundle $p\co \p M\to \beta M$. Assume that
$n$ is even, $\beta M$ is spin, and $M_\Sigma$ is spin$^c$ but not spin,
with spin$^c$ line bundle $\cL$ which is trivial over $M$ and pulled
back from $L_{\beta M}$ over $\beta M$ on the tubular neighborhood
$-D(L_{\beta M})$ of $\beta M$.  Then $\ind \Dirac_{(M_\Sigma, \cL)}\in K_n$
is an obstruction to existence of a conical metric of {\psc} in the sense of
\textup{Definition \ref{def:S1fib}}.
\end{theorem}
\begin{proof}
For simplicity we will drop mention of the line bundle $\cL$ as we are
always working with the line bundle outlined in the statement of the
theorem.  Suppose $M_\Sigma$ has a conical metric of {\psc}.  On $M$
itself, $\Dirac_{M_\Sigma}$ is just the usual Dirac operator on $M$,
and its square is just $\nabla^*\nabla + \frac14\kappa_M$, which is
bounded away from $0$.  On $T$ (a tubular
  neighborhood of $\beta M$), the spin$^c$ line bundle enters also,
and we get instead $\nabla^*\nabla + \frac14 \kappa_T + \mathcal R_L$,
where the term $\mathcal R_L$ has the form
\begin{equation*}
\mathcal R_L =\tfrac12 \sum_{j<k} F_L(e_j,e_k)\cdot e_j\cdot e_k,
\end{equation*}
where one sums over an orthonormal frame and
$F_L$ is the curvature of the line bundle $L$. Since the index
is invariant under a homotopy of the metric, we can replace the
original conical metric on $T$ by the restriction of a metric
on the $\bC\bP^1$-bundle compactification $\bP(L\oplus 1)$
of the line bundle $L$ over $\beta M$, smoothed out near $\partial M$
to patch with a psc metric on $M$,
and by making the $\bC\bP^1$ fibers very small,
we can get the scalar curvature of $T$ to dominate the
$\mathcal R_L$ term.  (This is exactly the same argument that was
used in \cite{MR1857524}.) Hence the square of the spin$^c$ Dirac
operator can be bounded strictly away from $0$ everywhere,
and when $(M, \p M)$ admits a conical metric of {\psc}, the $KU$-index
of $\Dirac_{M_\Sigma}$ must vanish, which is what we wanted to show.
\end{proof}
\begin{remark}
\label{rem:cylvscone}
We believe that it should be possible to use Theorem
\ref{thm:complexDiracobstr} to construct a spin manifold with
$S^1$-fibered singularities that admits a cylindrical psc metric, but
not a conical psc metric. This would show that 
the conditions of Definitions \ref{def:cyl-metric} and \ref{def:conicalmetric}
are not equivalent.  The idea would be to construct $M^n$, a compact
simply connected spin manifold, such that $\p M$ admits a free $S^1$-action
of \emph{even} type (in the sense of \cite{MR3449263}),
so that $\p M/S^1 = \beta M$ is spin and simply connected.
The action being of even type forces
the Chern class $c_1(L)$ classifying the line bundle $L$ and the circle bundle
$p\co\p M\to \beta M$ to be \emph{odd} (not to reduce to $0$ mod $2$).
Then $M_\Sigma$ would not be spin, and Theorem \ref{thm:complexDiracobstr} is
applicable.  If $n\ge 7$ and $\alpha(\beta M)=0$ in $KO_{n-2}$,
then $\beta M$ admits a metric of {\psc} by \cite{MR1189863}.
If one could choose this psc-metric so that its lift to an $S^1$-invariant
metric on $\p M$ extends over $M$, then $(M,\p M)$ would admit a cylindrical
psc-metric.  On the other hand, if $\alpha^c(M_\Sigma)\ne 0$,
then $M_\Sigma$ would not admit a conical psc-metric.
Unfortunately we have not yet been able to construct a case where we
can verify all these details simultaneously, so at
the moment, that cylindrical and conical psc-metrics are not equivalent
is only a conjecture.
\end{remark}

\subsection{\texorpdfstring{The case of $P=S^1$: bordism theory}{The case of P=S1: bordism theory}}
\label{sec:S1bord}
Here we want to consider the analogue of the
definition from Section \ref{sec:Zkbord}.
\begin{definition}
\label{def:S1bord}
Let $X$ be a topological space. The group $\Omega^\spS(X)$ will
consist of equivalence classes of maps $f_0\co M_\Sigma\to X$,
where $M_\Sigma$ is the closed manifold associated to an
$n$-dimensional spin manifold $(M,\partial M)$
with $S^1$-fibered singularities in the sense of Definition
\ref{def:S1fib}.  Two such maps
$f_0\co M_\Sigma\to X$ and $f_1\co M'_\Sigma\to X$ are said to be
equivalent if there is a spin
bordism $\overline M $ between $M$ and $M'$ as spin
manifolds with boundary, with a principal
$S^1$-bundle $\overline p \co \p \overline M \to \beta \overline M$
given by a free action of $S^1$ on $ \p \overline M$,
such that the restrictions $\overline p|_{\p M}$
and $\overline p|_{\p M'}$ coincide with the corresponding maps
$p \co \p M \to \beta M$ and $p' \co \p M' \to \beta M'$, and there is a map
$\overline f\co \overline M \to X$ restricting to $f_0$ and $f_1$
on $M$ and $M'$.  In particular, the manifold
$\beta \overline M$ gives a spin bordism of closed manifolds
$\beta \overline M \co \beta M \rightsquigarrow \beta M'$. \qed
\end{definition}
Exactly as in \cite{MR0346824}, it is easy to see that
$\Omega^\spS_*(\text{---})$ is a homology theory represented by a spectrum
$\MSpin^{S^1\fb}$.  We need the analogues of \eqref{eq:bordtriZk}
and of Proposition \ref{prop:BZkcofib}, but this is somewhat more
involved since, as we have seen in examples, $\beta M$ is not always
spin.  Thus the ``Bockstein map'' $\beta$ will take its values
in a direct sum of two bordism groups, one spin and one spin$^c$.
In fact (if $\p M$ has multiple components), $\beta M$ can have
components in both groups.

We have three natural transformations here. The first one 
\begin{equation*}
i \co \Omega^{\spin}_*(\text{---})\to \Omega^\spS_*(\text{---})
\end{equation*}
is given by considering closed spin manifolds as manifolds with empty
$S^1$-fibered singularity.

We notice that the Bockstein operator $M \mapsto \beta M$
comes together with a map $f\co \beta M\to BS^1=\bC\bP^\infty$
classifying the $S^1$-bundle $\p M \to \beta M$,
or equivalently a line bundle $L$ over $\beta M$. 
There is a natural splitting of the Bockstein
transformation into two pieces, $\beta_{\text{even}}$ and
$\beta_{\text{odd}}$. This splitting can be traced back to
\cite[Propositions 2.2 and 2.3]{MR882700}.
The two summands keep track of the components $\p_i M$ of
$\p M$ where the $S^1$-action is of \emph{even type} (preserves the
spin structure) and those where it is of \emph{odd type} (does not
preserve the spin structure), respectively. In both cases, we get a
pair $[N, L]$, where $N$ is one of the components $\beta_i M$ of
$\beta M$ and $L$ is a line bundle on $N$. However,
there is a difference:
\begin{enumerate}
\item[$\beta_{\text{even}}$:] In the even case, the
spin structure on $\p M$ descends to a spin structure on $N$,
so that we can think of $[N, L]$ as living in 
$\Omega^{\spin}_*(\bC\bP^\infty)$. 
\item[$\beta_{\text{odd}}$:] In the odd case, the manifold  $N$ has a
spin$^c$ structure determined by $L$, since in order to get a
spin structure on $\p_i M$, the class $w_2(\beta_i M)$ has to be
in the kernel of $p^*$, which means 
$c_1(L\vert_{\beta_i M})\equiv w_2(\beta_i M)\ \mod 2$
and $L\vert_{\beta_i M}$ determines
a spin$^c$ structure on $\beta_i M$. Thus in the odd case,
$[N, L]$ lives in $\Omega^{\spinc}_*$.  
\end{enumerate}

To sum up, we have a natural transformation
\begin{equation*}
  \Omega^\spS_*(\text{---})
  \xrightarrow{\beta = (\beta_{\text{even}},\,\beta_{\text{odd}})}
  \Omega^{\spin}_*(\text{---}\times \bC\bP^\infty)\oplus
  \Omega^{\spinc}_*(\text{---})
\end{equation*}
of degree $-2$. Finally we have a transformation
\begin{equation*}
  \tau \co \Omega^{\spin}_*(\text{---}\times \bC\bP^\infty) \oplus
  \Omega^{\spinc}_*(\text{---})
  \to  \Omega^{\spin}_*(\text{---})
\end{equation*}
of degree $+1$ which is a transfer: it sends a map $f\co N\to X$ to
$\tf\co \tN \to X$, where $\tN$ is the total space of the
$S^1$-bundle determined by the line bundle over $N$ given
by either the map to $\bC\bP^\infty$ or by the spin$^c$ structure.
\begin{theorem}
\label{thm:S1bordismseq}
We have an exact triangle of {\lp}unreduced{\rp} bordism theories
\begin{equation}
  \begin{diagram}
    \setlength{\dgARROWLENGTH}{1.95em}
  \node{\Omega^{\spin}_*(\text{---})}
          \arrow[2]{e,t}{i}
  \node[2]{\Omega^\spS_*(\text{---})}
          \arrow{sw,t}{\beta=(\beta_{\text{even}},\,\beta_{\text{odd}})}
  \\
  \node[2]{\Omega^{\spin}_*(\text{---}\times \bC\bP^\infty) \oplus
     \Omega^{\spinc}_*(\text{---})}
  \arrow{nw,t}{\tau}
  \end{diagram}.
  \label{eq:bordtriS1}
\end{equation}
\end{theorem}
\begin{proof}
As this is all we really need, we will show that
\begin{equation}
\label{eq:S1bordismseq}  
\cdots \xrightarrow{\beta} \Omega^{\spin}_{n-1}(\bC\bP^\infty)
\oplus\Omega^{\spinc}_{n-1} \xrightarrow{\tau} \Omega^{\spin}_n
\xrightarrow{i} \Omega^{\spS}_n \xrightarrow{\beta}
\Omega^{\spin}_{n-2}(\bC\bP^\infty) \oplus \Omega^{\spinc}_{n-2}
\xrightarrow{\tau} \cdots
\end{equation}
  is exact.  First we note a few easy facts: $\beta\circ i = 0$, since
  the image of $i$ consists of (classes of) closed manifolds with
  empty boundary; $\tau\circ \beta = 0$, since $\tau\circ \beta([M,\p
    M]) = [\p M]$, and $\p M$ bounds $M$; and $i\circ \tau = 0$, since
  if $M \xrightarrow{p} N$ is a principal $S^1$-bundle with associated
  complex line bundle $L$ over $N$ and $N$ spin$^c$ or spin and $M$
  spin, then $i\circ \beta([N, L]) = [M]$.  Let $W = M\times [0, 1]$
  with $\beta W = N$.  (See the comments in the parallel argument in
  the proof of Proposition \ref{prop:BZkcofib}.)

  Next, let's show that $\ker \tau\subseteq\image \beta$.  Suppose we
  have a spin manifold $N_1$, with line bundle $L_1$,
  and a spin$^c$ manifold $N_2$ with line bundle $L_2$ associated to
  the spin$^c$ structure.  Then $\tau([N_1, L_1]+[N_2, L_2]) = 0$ means that
  the disjoint union $M_1\amalg M_2$ is a spin boundary, where $M_j$ is
  the circle bundle over $N_j$ defined by $L_j$.  Let $P$ be a spin
  manifold with boundary $M_1\amalg M_2$.  Then $P$ is a spin manifold
  with $S^1$-fibered singularities and $\beta(P,\partial P) = [N_1,
    L_1]+[N_2, L_2]$.

  Let's show that $\ker\beta\subseteq\image i$.  Suppose $\beta([M,\p
    M]) = 0$.  That means $(\beta M, L)$ bounds in the appropriate
  sense, i.e., if $\beta M=\beta_1 M \amalg\beta_2 M$ with $\beta_1 M$
  spin and $\beta_2 M$ spin$^c$, then $[\beta_1 M, L_1]=0$ in
  $\Omega^{\spin}_{n-2}(\bC\bP^\infty)$ and $[\beta_2 M]=0$ in
  $\Omega^{\spinc}_{n-2}$.  Then changing things up to
  bordism, we can assume $\beta M$ is actually empty, and then clearly
  $M$ lies in the image of $i$.

  Finally, we show that $\ker i\subseteq\image \tau$. Suppose $M$ is a
  closed spin $n$-manifold and $i([M])=0$. That means that $M$ bounds,
  not necessarily in the sense of spin manifolds, but in the sense of
  spin manifolds with $S^1$-fibered singularities.  Choose $\overline
  M $ bounding $M$, with a free $S^1$-action on $\p\overline M $. This
  means in particular that $M$ is the total space of a principal
  $S^1$-bundle over a spin or spin$^c$ manifold, so $[M]$ lies in the
  image of $\tau$.
\end{proof}
\begin{lemma}
  \label{lem:S1cofib}
  Split the domain of the transfer map $\tau$ of
  \textup{\eqref{eq:S1bordismseq}} into three summands:
  \[
  \Omega^{\spin}_*\oplus \widetilde\Omega^{\spin}_*(\bC\bP^\infty)
  \oplus \Omega^{\spinc}_*.
  \]
  Then $\tau$ vanishes identically on $\Omega^{\spinc}_*$ and
  on $\widetilde\Omega^{\spin}_*(\bC\bP^\infty)$, and
  its restriction to $\Omega^{\spin}_*$ has finite image.
  There is a natural splitting of $\beta$ on the 
  $\Omega^{\spinc}_{n-2}$ summand.
\end{lemma}
\begin{proof}
  First suppose $N$ is a spin$^c$ manifold and $L$ is the line bundle
  defined by the spin$^c$ structure.  Then the total space of $L$ is a
  spin manifold, as explained in Remark \ref{rem:spinstr} and in
  Example \ref{ex:spinc}.  Hence the disk bundle $T(L)$ of $L$ is a
  spin manifold bounding the circle bundle of $L$, which is
  $\tau(N,L)$, and so $\tau([N,L])$ is trivial in spin bordism.
  (Note that the $S^1$-action on $\tau([N,L])$ is free of odd type.)
  Mapping $N$ back to $(-T(L),-\partial T(L))$ gives us a splitting of
  $\beta$ from $\Omega^{\spinc}_{n-2}$ to $\Omega^{\spS}_n$.

  The other summand in the domain of $\tau$ is
  $\widetilde\Omega^{\spin}_{n-2}(\bC\bP^\infty)$.  
  We can describe the transfer map $\tau$ on this summand with the
  aid of \cite[\S6]{Boardman}. (We are indebted to \cite{MR1189863}
  for this reference.)  This comes from a Thom map
  $\Sigma \bC\bP^\infty\to \bS$, where $\bS$ is the sphere spectrum,
  which comes from stabilizing the map
  $\Sigma\bC\bP^n\to S^{2n+1}$ collapsing all but the very top cell.
  Since this map is null-homotopic (in the limit as $n\to\infty$),
  the transfer map $\tau$ vanishes.
  
  Note that it was asserted
  in \cite[Proposition, p.\ 354]{MR0248858} that
  $\widetilde\Omega^{\spin}_{n-2}(\bC\bP^\infty)$ is isomorphic
  to $\Omega^{\spinc}_n$, but this appears to be a misprint, since,
  for example, $\Omega^{\spinc}_0\cong \bZ$ and
  $\widetilde\Omega^{\spin}_{-2}(\bC\bP^\infty)=0$.  The correct
  statement may be found in \cite[Remark 5.10]{MR1857524},
  that there is an equivalence
  $\MSpinc\simeq \MSpin\wedge \Sigma^{-2}\bC\bP^\infty$, or in other
  words $\Omega^{\spinc}_n\cong
  \widetilde\Omega^{\spin}_{n+2}(\bC\bP^\infty)$.
  (The geometric interpretation of this is that dualizing a complex
  line bundle over a closed spin manifold gives a codimension-$2$
  spin$^c$ submanifold.)
  So $\widetilde\Omega^{\spin}_{n-2}(\bC\bP^\infty)\cong \Omega^{\spinc}_{n-4}$. 
  \end{proof}  
Restating Theorem \ref{thm:S1bordismseq} and Lemma \ref{lem:S1cofib}
in slightly different language, we have the following:
\begin{proposition}
  The bordism spectrum $\MSpin^{S^1\fb}$ sits in a cofiber
  sequence, which is rationally split, of the form
  \[
  \MSpin \to \MSpin^{S^1\fb} \to \Sigma^2\MSpin \vee \Sigma^2\MSpin^c
  \vee \Sigma^4\MSpin^c.
  \]
\label{prop:S1cofib}
\end{proposition}
\begin{proof}
  This follows from the homotopy-theoretic description of the bordism
  groups and Lemma \ref{lem:S1cofib}.
\end{proof}
\begin{definition}
\label{def:S1K}
The orientation maps $\alpha\co\MSpin \to \KO$ and $\alpha^c\co\MSpin^c \to \K$
for spin and spin$^c$ manifolds respectively are 
defined by taking the index of the appropriate Dirac operator. Then the
maps $\alpha$ and $\alpha^c$ give a map of spectra
\begin{equation}\label{alpha-s1}
\ind^{S^1\fb}\co \MSpin^{S^1\fb}\to \Sigma^2\KO \vee \Sigma^2\K.
\end{equation}
In more detail, the $\K$ component corresponds to the $\alpha^c$
invariant for $M_\Sigma$, and the $\KO$ component corresponds to the
$\alpha$ invariant for spin components of $\beta M$.  In addition, the
relative index of \textup{Proposition \ref{prop:relat-ind}} is defined
on the kernel of $\ind^{S^1\fb}$. Indeed, when $\ind^{S^1\fb}(M,\p M)=0$,
$\beta M$ admits a metric of positive scalar curvature, and then
by Corollary \ref{cor:pscconical} there is a metric of positive scalar
curvature on $\p M$ extending over a tubular neighborhood of $\beta M$,
and we can compute the index obstruction to extending this metric
over $M$.  The relative index might depend on the choice of a positive
scalar curvature metric on $\beta M$. More specifically, fixing a path
component of positive scalar curvature metrics on $\beta M$ is
essentially equivalent (via the correspondence in \cite{BB}) to fixing
a path component of the $S^1$-invariant psc-metrics on $\partial M$.
Between any two of these, there is a relative index in $KO_n$, $n=\dim
M$, and the relative index obstruction for extending an
$S^1$-invariant psc-metric $g'_\p$ on $\p M$ over $M$ is the same as
the relative index obstruction for another $S^1$-invariant psc-metric
$g_\p$ plus the relative index $i(g'_\p,g_\p)$, by \cite[Theorem
  4.41]{MR720933}.
\end{definition}
\begin{theorem}[Cf.\ {\cite[Theorem 7.4(1)]{MR1857524}}]
\label{thm:S1bordism}
Let $M$ be a simply connected spin manifold with $S^1$-fibered
singularities, of dimension $n\ge 7$.  Assume that each component of
$\p M$ is simply connected {\lp}so that each component of
$\beta M$ is simply connected as well{\rp}.  If the class of $M$ in
$\Omega_n^{\spS}$ contains a representative $M'$ of {\psc}, then $M$
admits a metric of \psc.
\end{theorem}
\begin{proof}
As we saw in Lemma \ref{lem:S1cofib}
and Proposition \ref{prop:S1cofib}, $\Omega_n^{\spS}$ is built out
of spin and spin$^c$ bordism groups.  For the spin bordism groups,
the proof is essentially the same as that of {\cite[Theorem
  7.4(1)]{MR1857524}}, using spin surgeries either in the interior
of $M'$ or on $\beta M'$ (and then lifted to $\partial M'$). The
latter case requires a little bit of care, as we explain below.

Let $\bar M\co M \bord M'$ be a bordism in $\Omega_n^{\spS}$. In particular,
it means we have a bordism
$(\beta \bar M)\co (\beta  M,f)\bord (\beta M',f')$,
where $\bar f\co \beta \bar M \to  \CP^{\infty}$
is a classifying map for the $S^1$ fiber bundle
$\bar  p \co\p \bar M \to\beta \bar M$, which restricts to corresponding
classifying maps $f$ and $f'$ on the boundary
$\p (\beta \bar M)= \beta M \sqcup -\beta M'$.

In order to ``push'' a psc metric from $M'$ to $M$, we have to show
that the bordism $\bar M \co M \bord M'$ can be modified so that the
embeddings $(M',\p M') \hookrightarrow (\bar M, \p \bar M)$ and $\beta
M' \hookrightarrow \beta \bar M$ are both 2-connected. The only
non-standard part is a modification of the bordism
$\beta \bar M \co \beta M \bord \beta M'$.

First we notice that the inclusion $i'\co \beta M'\hookrightarrow\beta
\bar M$ can be assumed to be 1-connected. Indeed, let
$\iota\co S^1\hookrightarrow \beta \bar M\setminus \beta M'$
be an embedding;
then the composition $\bar f\circ \iota \co S^1 \hookrightarrow \beta
\bar M\to \CP^{\infty}$ is null-homotopic. Hence the classifying map
$\bar f$ can be extended to the manifold resulting from surgery along
$S^1\hookrightarrow\beta \bar M\setminus \beta M'$. Moreover, the same
argument shows that we can assume both manifolds $\beta M'$ and $\beta
\bar M$ are simply-connected.

The $S^1$-bundle $\bar p\co \p \bar M \to \beta \bar M$ restricts to the bundle
$p'\co \p M' \to \beta M'$; together they give us the commutative diagram:
\begin{equation}\label{eq:pi2}
\begin{diagram}
\setlength{\dgARROWLENGTH}{1.15em}
   \node{\pi_2 \p M'}
	\arrow{s,r}{j'_*}
        \arrow{e,t}{ p_*'}
   \node{\pi_2 \beta M'}
	\arrow{s,r}{i'_*}
        \arrow{e,t}{\p'}
   \node{\pi_1 S^1}
	\arrow{s,r}{\text{Id}}
        \arrow{e}        
        \node{0}
\\
 \node{\pi_2 \p \bar M}
        \arrow{e,t}{\bar p_*}
   \node{\pi_2 \beta \bar M}
        \arrow{e,t}{\bar \p}
   \node{\pi_1 S^1}
        \arrow{e}        
        \node{0}        
\end{diagram}
\end{equation}
Let $\phi\co S^2 \to \beta M'$ be such that
$\p'\co [\phi]\mapsto 1\in \pi_1(S^1)$.
Then $\bar \p \co [\bar \phi]\mapsto 1\in \pi_1(S^1)$,
where $\bar \phi = \phi'\circ i' \co S^2 \to \beta \bar M$.

Given these elements, we consider a map
$\bar \psi\co S^2 \to \beta \bar M$ such that $[\bar \psi] = a \neq 0$
in the factor group
$\pi_2\beta\bar M/\mbox{Im }i'_*$.  If $\bar \p ([\bar \psi])=0$, then
we can do surgery along $S^2\hookrightarrow\beta \bar M$ and extend
the map $\bar f$ to the resulting manifold. If
$\bar \p ([\bar \psi])= k\in \pi_1 S^1$,
we can adjust the element $[\bar \psi]$ by adding
$-k[\bar \phi]$. This creates a representative
$\tilde \psi\co S^2 \to \beta \bar M$
of the same element $a\in \pi_2\beta\bar M/\mbox{Im }i'_*$
such that $\bar \p [\tilde \psi]=0$. This construction gives a
bordism $\beta \bar M\co \beta M \bord \beta M'$, where the embedding
$i'\co \beta M'\hookrightarrow \beta \bar M$ is 2-connected. Finally,
to obtain the desired modification of the bordism $\bar M\co M \bord M'$, 
we perform surgeries on the interior of $\bar M\setminus \p \bar M$ using the
standard method found, say, in \cite[Theorem 7.4]{MR1857524}.  We have to use
the restriction $n\geq 7$ since the Bockstein manifolds $\beta M$ and
$\beta M'$ need to have dimension at least five.  For the spin$^c$
bordism groups, the proof will be given in Theorem
\ref{thm:bordismspinc}.
\end{proof}  

Now we get to the main part of Theorem B.  The following Theorem \ref{thm:S1psc}
relies on Corollary \ref{cor:nonspinpsc}, which will be proved in
Section \ref{sec:CP2}.
\begin{theorem}
\label{thm:S1psc}
Let $(M, \p M)$ be a spin $n$-manifold with $S^1$-fibered singularities,
with $n\ge 7$, determined by an $S^1$-bundle $p\co \p M\to \beta M$.
Assume that $M$ and $\p M$ are connected and simply connected
{\lp}so that $\beta M$ is connected and simply connected as well{\rp}.
Then $M$ admits a conical metric of {\psc} in the sense of
\textup{Definition \ref{def:S1fib}} if and only if the bordism class of
$M$ maps to $0$ in $KO_{n-2}\oplus K_{n-2}$ under
the map $\ind^{S^1\fb}$ from {\rm (\ref{alpha-s1})}  and the
relative index of \textup{Proposition \ref{prop:relat-ind}}
vanishes in $KO_n$ for some choice of a {\psc}
metric on $\beta M$.
\end{theorem}
\begin{remark}
\label{rem:etapsc}
Note that the case of $\eta$-type
singularities of \textup{\cite[Theorem 1.1]{MR1857524}}
is subsumed in \textup{Theorem \ref{thm:S1psc}}. In the case considered
there, $p$ is a trivial bundle and $\p M = \eta\times \beta M$,
where $\eta$ is $S^1$ with its non-bounding spin structure and
$\beta M$ is spin.  In that case, $M_\Sigma$ is spin$^c$ and the
obstruction constructed in the proof on pages 715--716 of
\cite{MR1857524} is the component of $\ind^{S^1\fb}$ with values in
$K_n$. Also in that case, because of the classical exact sequence
\[
\cdots \to KU_n \xrightarrow{\beta} KO_{n-2} \xrightarrow{\eta} KO_{n-1}
\to KU_{n-1} \to \cdots,
\]
vanishing of $KU_n$ index implies that $\alpha(\beta M)=0$ in
$KO_{n-2}$, so it wasn't necessary to worry about this component
of $\ind^{S^1\fb}$, and $\beta M$ automatically had {\psc}.
Similarly, in this case the relative index in $KO_n$ also
vanishes when the $KU_n$ index vanishes.
\end{remark}
\begin{proof}[Proof of Theorem \ref{thm:S1psc}]
We begin with the necessity of the condition, which doesn't involve
the dimensional restriction $n\ge 7$.  Vanishing of the $KO_{n-2}$
term when $\beta M$ is spin is needed for $\beta M$ to admit a metric
of {\psc}, which as we have seen is a necessary condition. Similarly,
vanishing of the $KO_n$ term when $M_\Sigma$ is spin is needed for
$M_\Sigma$ to admit a metric of {\psc}, which is also a necessary
condition.  If $\beta M$ is spin$^c$ but not spin, then the $KO_{n-2}$
index automatically vanishes, and $\beta M$ admits a metric of {\psc}.
Then by Corollary \ref{cor:pscconical}, with the right scaling
of the metric on the fibers, there is a conical metric
on the tubular neighborhood $T$ of $\beta M$ which also admits {\psc}.
Now the $KU_n$ index of the spin$^c$ Dirac operator $\Dirac_{M_\Sigma}$
on $M_\Sigma$ is also an obstruction to {\psc}, by Theorem
\ref{thm:complexDiracobstr}.

The argument for sufficiency uses the surgery method and the condition
$n\ge 7$, which guarantees that $M$ and $\beta M$ each have dimension
at least $5$, the range where the $h$-cobordism theorem applies.
Suppose that $\ind^{S^1\fb}(M,\p M)=0$ and that $M$ and $\p M$
{\lp}and hence also $\beta M${\rp} are simply connected.

First assume that the $S^1$-action on $\p M$ is of odd type.  
In this case we view $(\beta M, L)$ as defining a class in
$\Omega_{n-2}^{\spinc}$.  By Lemma \ref{lem:S1cofib}, we have
a splitting $\Omega_{n-2}^{\spinc}\to \Omega^{\spS}_n$ obtained by
taking the disk bundle $D(L)$, which is spin since
$w_2(D(L))=\widetilde p^*(w_2(\beta M)) = \widetilde p^*(c_1(L))\pmod 2=0$
by the Gysin sequence.  This admits a conical metric of {\psc},
so subtracting this off, we can assume up to bordism that
$\p M$ is empty.  Then $M$ is closed simply connected spin manifold
which has vanishing $\alpha$-invariant, since for a closed manifold,
the $\alpha$-invariant coincides with the relative index for extending
a psc-metric from the empty boundary.  Thus
M admits {\psc} by Stolz's Theorem \cite{MR1189863}.  So we're
done with this case by Theorem \ref{thm:S1bordism}.

Now assume that the
$S^1$-action on $\p M$ is of even type and thus that
$\beta M$ is spin.  Since $\ind^{S^1\fb}(M,\p M)=0$,
we have the condition $\alpha(\beta M)=0$, and since
$\beta M$ is simply connected, it admits a
metric of {\psc} by \cite{MR1189863}. Lift such a metric to an
$S^1$-invariant metric on $\p M$. Suppose the relative index for
extending this metric to a psc-metric over $M$ vanishes.
Since the $S^1$-action on $\p M$ is of even type,
the disk bundle $D(L)$ of the line bundle $L$ over $\beta M$ is \emph{not}
spin, since $S^1$ with the $S^1$-invariant spin structure is $\eta$, which
is not a spin boundary, and thus the disk bundle $D(L)$ is spin$^c$
but not spin. In this situation, 
the condition $\ind^{S^1\fb}(M,\p M)=0$ gives
vanishing of the $KU$-index of $\Dirac_{M_\Sigma}$.  In this case, we
need to use Corollary \ref{cor:nonspinpsc}, together with
Theorem \ref{thm:bordismspinc} applied to $M_\Sigma$ with the
spin$^c$ line bundle $\cL$ of Theorem \ref{thm:complexDiracobstr},
which is trivial on $M$ and is pulled back from $L$ on the tubular
neighborhood $D(L)$ of $\beta M$; in particular $\cL|_{\beta M}=L$.
Let us check that the hypotheses of Theorem \ref{thm:bordismspinc}
are satisfied.  The manifold $M_\Sigma = M\cup_{\p M} D(L)$ is simply
connected by Van Kampen's Theorem, since $\p M$ is connected and $M$
and $D(L)\simeq \beta M$ are simply connected, and it is not spin,
as we just saw.  Furthermore, under the Hurewicz identification
$\pi_2(\beta M)\cong H_2(\beta M)$, evaluation of $c_1(L)$ is the
boundary map $\pi_2(\beta M)\to \pi_1(S^1)\cong \bZ$ in the homotopy
exact sequence of the circle bundle $p\co \p M\to \beta M$, which is
surjective since $\pi_1(\p M)=0$.  As $c_1(\cL)$ restricts to
$c_1(L)$ on $\beta M$, it follows that
$c_1(\cL)\co H_2(M_\Sigma)\to \bZ$ is surjective, as required in
Theorem \ref{thm:bordismspinc}. (This is the only place in the paper
where Theorem \ref{thm:bordismspinc} is used, and it is here that
the assumption that $\p M$ is simply connected is needed.)
By the method of
proof, we can do spin$^c$ surgeries on $M_\Sigma$ away from
the interior of $M$, possibly changing $\beta M$ and the line bundle
$L$ over it, but preserving {\psc} and
the condition $\frac14 \kappa_g + \cR_h > 0$.
This has the effect of changing the class of $(M,\p M)$
to a new $(M', \p M')$ in the same bordism class in $\Omega^{\spS}_n$.
The new $\beta M'$ will be in the same spin bordism class as
$\beta M$, and we can arrange for it still to be simply connected.
Since we can keep the spin$^c$ line bundle and the correction term
$\cR_h$ trivial on $M'$, the new $M'$ will have a conical metric
of {\psc}. Again we conclude by Theorem \ref{thm:S1bordism}.
\end{proof}
\subsection{Proof of Theorem B}
\begin{remark}
\label{rem:reformulation}
Theorem \ref{thm:S1psc} can also be reformulated as the statement
labeled Theorem B in Section \ref{sec:intro}, using the map
$\alpha^{S^1\fb}$ obtained by naturality of the cofiber to get a
commutative diagram of spectra
\begin{equation*}
  \begin{diagram}
    \setlength{\dgARROWLENGTH}{1.95em}
\node{\Sigma (\MSpin\wedge \CP^{\infty}_+)\vee \Sigma \MSpin^c}
        \arrow{e,t}{\tau}
        \arrow{s,t}{\Sigma \alpha\vee \Sigma \alpha^c}
\node{\MSpin}
        \arrow{e,t}{i}
        \arrow{s,t}{\alpha}
\node{ \MSpin^{S^1\fb}}
        \arrow{s,t}{\alpha^{S^1\fb}}
\\
\node{\Sigma (\KO\wedge \CP^{\infty}_+)\vee \Sigma \KU}
        \arrow{e,t}{\tau^{\KO}}
\node{\KO}
        \arrow{e,t}{i^{\KO}}
\node{\,\KO^{S^1\fb}.}
  \end{diagram}
\end{equation*}
\end{remark}
\begin{remark}
\label{rem:disconnected}
Careful examination of the proof above shows that we don't really
need to assume that $\p M$ and $\beta M$ are connected.  The theorem still
applies when they have more than one component, as long as each
component of $\p M$ {\lp}and hence each component of $\beta M${\rp}
is simply connected.  In this situation, it is possible
to have ``mixing'' of the various cases, as it is possible for
some components of $\p M$ to be of even type and for other components
to be of odd type.
\end{remark}
\section{\texorpdfstring{A spin$^c$ surgery theorem}{A spinc surgery theorem}}
\label{sec:surg}
In this section we adapt the methods of Gromov and Lawson
\cite{MR577131} and Hoelzel \cite{MR3498902} to the situation needed
in Section \ref{sec:S1}.  Recall that to show that a manifold with
$S^1$-fibered singularities admits a conical metric of {\psc},
sometimes we need to deal with the effect of surgery on the Bochner
term $\frac14 \kappa_{\beta M} + \mathcal R_L$ in the square of the Dirac
operator on a spin$^c$ manifold $M$ with the line bundle $L$
determined by the spin$^c$ structure.  For the proof of Theorem
\ref{thm:S1psc}, we need to know that under some circumstances,
positivity of this term is preserved by spin$^c$ surgeries (which of
course can change the line bundle $L$).  This will provide a partial
converse to \cite[Corollary D.17]{lawson89:_spin}, and show that under
some circumstances, existence of a {\psc} metric on $M$ together
with a hermitian metric and connection on $L$ for which
$\frac14 \kappa_{M} + \mathcal R_L > 0$ is
\emph{equivalent} to vanishing of $\alpha^c(M,L)$.

We need to consider coupling between the Riemannian
curvature and the curvature of a line bundle $L$ (which is just given
by an ordinary $2$-form $\omega$, which after dividing by $2\pi\,i$, has
integral de Rham class representing $c_1(L)$). Now recall
\cite[Lemma D.13]{lawson89:_spin}, which says that \emph{any}
$2$-form $\omega$ with $\frac{\omega}{2\pi\,i}$ in the de Rham class of
$c_1(L)$ can be realized as the curvature of some unitary connection
on $L$. To state the next theorem we need to recall the definition
of spin$^c$ surgery.
\begin{definition}
Let $(M, L)$ be a closed spin$^c$ manifold (so $M$ is a closed
oriented manifold and $L$ is a complex line bundle on $M$ with
$c_1(L)$ reducing mod $2$ to $w_2(M)$). We say that
$(M', L')$ is obtained from $(M,L)$ by \emph{spin$^c$ surgery in codimension
$k$} if there is a sphere $S^{n-k}$ embedded in
$M$ with trivial normal bundle, $M'$ is the result of gluing
in $D^{n-k+1}\times S^{k-1}$ in place of $S^{n-k}\times D^k$, and
there is a spin$^c$ line bundle $\cL$ on the
trace of the surgery---a bordism from $M$ to $M'$---restricting
to $L$ on $M$, and the bundle $L'$ on $M'$ is the restriction of $\cL$. 
\label{def:spincsurg} 
\end{definition}
The following theorem deals with ``twisted scalar curvature''
on a spin$^c$ manifold, the twisting coming from the curvature of the
spin$^c$ line bundle, as advertised in the Abstract of this paper.
\begin{theorem}[Spin$^c$ surgery theorem]
\label{thm:codim3spinc}
Let $M^n$ be a closed $n$-dimensional spin$^c$ manifold, with line bundle
$L$ over $M$ defining the spin$^c$ structure.  Assume that $M$ admits a
Riemannian metric $g$ and $L$ admits a hermitian bundle metric $h$
and a unitary connection such that $\frac14 \kappa_{M} + \mathcal R_L > 0$
{\lp}in the notation of the proof of \textup{Theorem \ref{thm:S1psc}}{\rp}.
Let $(M', L')$ be obtained from $(M,L)$ by spin$^c$
surgery in codimension $k\ge 3$, as in Definition
\textup{\ref{def:spincsurg}}.
Then there is a metric $g'$ on $M'$, and $L'$ admits a hermitian bundle
metric $h'$ and a unitary connection, such that
$\frac14 \kappa_{M'} + \mathcal R_{L'} > 0$.
\end{theorem}
This leads to the following bordism theorem, which was needed in the proof
of Theorem \ref{thm:S1psc}.
\begin{theorem}[Spin$^c$ bordism theorem]
\label{thm:bordismspinc}
Let $M^n$ be a connected closed $n$-dimensional spin$^c$ manifold
\textbf{which is not spin}, with line bundle
$L$ over $M$ defining the spin$^c$ structure.  Assume that $M$ is simply
connected and that $n\ge 5$.  \textbf{Also assume that
$c_1(L)\co H_2(M)\to\bZ$ is surjective,
or equivalently, that the classifying
map $M\to \BSpinc$ of $(M,L)$ is a $2$-equivalence 
{\lp}an isomorphism on $\pi_1$ and a 
surjection onto $\pi_2(\BSpinc)=\bZ${\rp}.}
Finally, assume that there exists
a pair $(M',L')$ in the same bordism class in $\Omega^{\spinc}_n$
with a metric $g'$ on $M'$ and a hermitian metric $h'$ and
a unitary connection on $L'$ such that
$\frac14 \kappa_{M'} + \mathcal R_{L'} > 0$.  Then
$M$ admits a Riemannian metric
$g$ and $L$ admits a hermitian bundle metric $h$
and a unitary connection such that $\frac14 \kappa_{M} + \mathcal R_L > 0$.
\end{theorem}  
\begin{proof}[Proof of \textup{Theorem \ref{thm:bordismspinc}} assuming
    \textup{Theorem \ref{thm:codim3spinc}}]
The following argument is adapted from \cite{MR577131}.
First of all, by Theorem \ref{thm:codim3spinc} we can kill off $\pi_0$
and $\pi_1$ of $M'$ and assume that $M'$ is simply connected.
Let $W$ be a spin$^c$ manifold with spin$^c$ line bundle $\cL$ and
with $\partial W = M \amalg -M'$, $\cL|_{M'}=L'$, $\cL|_{M}=L$.
Doing spin$^c$ surgery on $W$ if necessary, we can also assume
that $W$ is simply connected.  In order for $M$ to be obtainable
from $M'$ by spin$^c$ surgery in codimension $\ge 3$, we need the
inclusion map $M\hookrightarrow W$ to be a $2$-equivalence, i.e.,
to induce an isomorphism on $\pi_0$ and $\pi_1$ and a surjection on $\pi_2$.
The $\pi_0$ and $\pi_1$ conditions certainly hold since $M$ and $W$ are
simply connected. By the Hurewicz Theorem, we may identify $\pi_2(M)$
and $\pi_2(W)$ with $H_2(M)$ and $H_2(W)$, respectively, and thus
we need to be able to kill $H_2(W, M)$ by surgery.
Now since $n\ge 5$, classes in $\pi_2(W)$ are
represented by smoothly embedded $2$-spheres.  The normal bundle of
such a $2$-sphere $S$ is determined by $\langle w_2(W), [S]\rangle$,
where we view $w_2(W)$ by the Hurewicz Theorem and Universal
Coefficient Theorem as a map from $\pi_2(W)$  to $\bZ/2$, and the
spin$^c$ condition says that $c_1(\cL)$ reduces mod $2$ to $w_2(W)$,
and similarly $c_1(L)$ reduces mod $2$ to $w_2(M)\ne 0$.  Thus the normal
bundle of $S$ is trivial whenever $\langle c_1(\cL), [S]\rangle=0$.
However, one has to be careful here: in order to do \emph{spin$^c$}
surgery on $S$ in the interior of $W$, i.e., surgery such that $\cL$
extends over the trace of the surgery, so that the result is again a
spin$^c$ bordism between $(M,L)$ and $(M',L')$, it is not enough for
the normal bundle of $S$ to be trivial; one needs
$\langle c_1(\cL), [S]\rangle=0$ (and not just
$\langle w_2(W), [S]\rangle=0$), exactly as in the last paragraph
of the proof of Theorem \ref{thm:codim3spinc} below.  Thus the
spin$^c$ surgeries on $2$-spheres in the interior of $W$ that we are
allowed to do are exactly those on spheres representing classes in
the kernel of $c_1(\cL)\co H_2(W)\to \bZ$, and doing them kills this
kernel.  After this, $c_1(\cL)$ is injective on $H_2(W)$, so that
$H_2(W)$ is either $0$ or infinite cyclic. Now $c_1(L)=c_1(\cL)|_M$,
so the map $c_1(L)\co H_2(M)\to\bZ$ factors as
$H_2(M)\to H_2(W)\xrightarrow{c_1(\cL)} \bZ$.  Since $c_1(L)$ is
surjective by hypothesis, it follows that $H_2(W)\cong \bZ$ and that
$H_2(M)\to H_2(W)$ is surjective, hence that $H_2(W,M)=0$ by the
exact sequence of the pair $(W,M)$, since $H_1(M)=0$.  (Without
the hypothesis that $c_1(L)$ is surjective, one would only be able to
reduce $H_2(W,M)$ to a finite cyclic group, of odd order since
$c_1(L)$ is nonzero mod $2$, and this is not good enough; see
Propositions \ref{prop:counterexample1} and
\ref{prop:counterexample2} below.)
Then by the proof of the $h$-Cobordism Theorem,
using handle cancellations as in \cite{MR0190942}
or \cite[Lemma 1]{MR0189048},
$M$ can be obtained from $M'$ by spin$^c$ surgeries in codimension
$\ge 3$, and we can apply Theorem \ref{thm:codim3spinc}.
\end{proof}
\begin{proof}[Proof of \textup{Theorem \ref{thm:codim3spinc}}]
Let $N^{n-k}\hookrightarrow M^n$ be the embedded sphere on which we
are doing surgery.  Recall the proof of the surgery theorem
in \cite{MR577131} or in \cite{MR3498902}: we ``bend'' the metric
$g$ on a neighborhood $N^{n-k}\times D^k$ of $N$, preserving the
property of {\psc}, so that close to
$N$ the metric looks like the standard metric on
$S^{n-k}(r_1)\times S^{k-1}(r_2)\times [0, 1]$, for suitable $r_1,r_2>0$.
Then it is clearly possible to glue in a handle isometrically,
and since $k-1\ge 2$ and thus the second factor has {\psc}, we get
{\psc} on the new manifold $M'$.  In the course of this process,
up to an error which can be taken arbitrarily small, the
scalar curvature \emph{only goes up} from $\kappa_M$ to the scalar curvature
of $S^{n-k}(r_1)\times S^{k-1}(r_2)$, and thus during the bend, the
property that $\frac14 \kappa_{M} + \mathcal R_L > 0$ is preserved.
(The estimates checking this are rather involved, but are explained in
great detail in \cite[\S\S2.3--2.6]{MR2789750}.)  Now we
can proceed as follows: once the bend is completed, in a still smaller
neighborhood of $N$, we can change the curvature $2$-form $\omega$ of $L$
by any exact form, and we can choose this exact form to kill off
$\mathcal{R}_L$ altogether in this small neighborhood, except in
the case where $n-k=2$ and $\langle c_1(L), [N]\rangle$ is nonzero.
Fortunately this case can never arise, because if $N=S^2$
and $L|_N$ is topologically nontrivial, then since the trace of the
surgery $W^{n+1}$ is obtained by attaching a $3$-handle to kill $[N]$,
the boundary map $H^2(M)\to H^3(W, M)$ has to kill $c_1(L)$, which
shows that $L$ cannot be the restriction of a line bundle over $W$.

Thus we can assume $\mathcal{R}_L$ vanishes in a small neighborhood of
the sphere on which we are doing the surgery, and then gluing in a
standard handle preserves the curvature condition.
\end{proof}
\begin{remark}
Note the careful way in which Theorem \ref{thm:bordismspinc} is written.
One should \emph{not} assume that if $(M,L)$ is a spin$^c$ and
$\alpha^c(M,L)=0$, then one can choose a metric on $M$ and a hermitian
metric and connection on $L$ so that $\frac14 \kappa_{M} + \mathcal R_L > 0$,
for this is false.  Indeed, suppose $M$ is actually spin and $L$
is the trivial bundle, so that $\alpha^c(M,L)=0$ means that
$\widehat A(M)=0$.  This is not enough for $M$ to admit {\psc},
even with $\dim M>4$ and $M$ simply connected,
because if $M$ has dimension $1$ or $2$ mod $8$, then $\alpha(M)$ might
be a non-zero $2$-torsion class.  Adding in the term $\mathcal R_L$
in this case only makes things worse, because in suitable
coordinates, $\mathcal R_L$ has the form
$\begin{pmatrix} \omega &0\\0&-\omega\end{pmatrix}$,
where the operator $\omega$ is constructed from the curvature of $L$,
which can be any exact $2$-form on $M$, so  $\frac14 \kappa_{M} + \mathcal R_L$
cannot be strictly positive in this case unless $\kappa_{M}$ is
strictly positive, which is impossible.
\end{remark}  

The hypothesis in Theorem \ref{thm:bordismspinc} that
$c_1(L)\co H_2(M)\to \bZ$ be surjective also cannot be dropped: without
it, the conclusion of the theorem can fail, even for simply connected
non-spin $M$.  To explain this, it is convenient to use the
terminology introduced in \cite{MR4817674}.  For a spin$^c$ manifold
$M$ with spin$^c$ line bundle $L$, a Riemannian metric $g$ on $M$,
a hermitian metric on $L$ and a unitary connection $A$ on $L$,
the \emph{generalized scalar curvature} of $(g,A)$ is
$\Rgen_{(g,A)}:=\kappa_g-2|\Omega|$, where $\kappa_g$ is the scalar
curvature of $g$, $\Omega$ is the curvature $2$-form of $A$, and
$|\Omega|$ is its pointwise operator norm (for the action on spinors).
Positivity of $\Rgen_{(g,A)}$ is equivalent \cite[Lemma 3.1]{MR4817674}
to the condition $\frac14 \kappa_{M} + \mathcal R_L > 0$ used above.
We say that $(M,L)$ admits \emph{generalized positive scalar curvature}
(gpsc) if there exist $g$, a hermitian metric on $L$ and $A$ with
$\Rgen_{(g,A)}$ strictly positive. Recall that gpsc implies
$\alpha^c(M,L)=0$, where
\[
\alpha^c(M,L)=\left\langle \Ahat(M)e^{c_1(L)/2},
[M]\right\rangle
\]
is the index of the spin$^c$ Dirac operator, a spin$^c$ bordism
invariant of the pair $(M,L)$.  (In general, it depends
on the choice of spin$^c$ structure through the
choice of the line bundle $L$.  For simply
connected manifolds, or more generally, manifolds
with no $2$-torsion in $H^2$, the choice of
spin$^c$ structure is equivalent to the choice
of spin$^c$ line bundle.)
First we need a rather easy lemma.
\begin{lemma}
\label{lem:obstr}
Let $(M,L)$ be a spin$^c$ manifold with
$L=E^{\otimes k}$ {\lp}$k\ge 1${\rp}, 
a tensor power of another
spin$^c$ line bundle $E$.  Assume that
$(M,L)$ admits generalized positive scalar
curvature.
Then $(M,E^{\otimes j})$ admits generalized 
positive scalar curvature for any $j$ with
$|j|\le k$.
\end{lemma}
\begin{proof}
First of all, as noted in \cite[\S3]{MR4817674},
the condition doesn't depend on the hermitian metric
on the bundle, so we can forget about that.
Also, unitary connections on a line bundle form an 
affine space over
$i\Omega^1(M)$, and $A_E\mapsto A_E^{\otimes k}$ 
is affine with
linear part multiplication by $k$; hence the given connection $A$ on $L$
is $A_E^{\otimes k}$ for a unique unitary 
connection $A_E$ on
$E$. The curvature $2$-forms are related
by $\Omega_L= k\Omega_E$, so
\[
\Rgen_{(g,A)} = \kappa_g - 2k|\Omega_E|
\le \Rgen_{(g,A_{E^{\otimes j}})} = \kappa_g - 2|j|\,
|\Omega_E|,
\]
and the result follows immediately.
\end{proof}
\begin{proposition}
\label{prop:counterexample1}
Let $Q^6$ be a smooth complex quadric hypersurface in
$\CP^4$, viewed as a smooth $6$-manifold,
and let $M^6=Q^6\# (K^4\times S^2)$,
where $K$ is a K3 surface {\lp}viewed as a 
smooth $4$-manifold{\rp}.  Then 
$H^\bullet(Q;\bZ)\cong 
\bZ[h,l]/(h^2-2l, l^2)$, $h$ the hyperplane
section class in $H^2(Q; \bZ)$, $l$
the class in $H^4(Q; \bZ)$ of a maximal
linear subspace of $Q$,
and $w_2(Q)$ is the mod-$2$ reduction of $h$.  
If $z$ is the usual generator
of $H^2(S^2)$, then we can also view $h$ and $z$
as living in $H^2(M)$, with $hz=0$ and $z^2=0$,
and $ah+bz$ defines a spin$^c$ structure
if $a$ is odd and $b$ is even, with
$\alpha^c(M,ah+bz)=\frac{a(a^2-1)}{24}+b$
{\lp}for $a$ odd, $b$ even{\rp}.
\end{proposition}
\begin{proof}
Since $\alpha^c$ is additive on connected sums,
$\alpha^c(M,ah+bz)=\alpha^c(Q,ah)
+\alpha^c(K\times S^2,bz)$.  The structure of the cohomology ring of $Q$
can be found in \cite[Ch.\ 6]{MR1288523}
and in \cite{quadrics}.
The $\alpha^c$ invariant of $Q$
is effectively computed in \cite{MR684651}.
The second summand is
\[
\left\langle \Ahat(K\times S^2)e^{bz/2}, [M]\right\rangle = 
\langle\Ahat(K),[K]\rangle
\langle bz/2,[S^2]\rangle
= 2b/2 = b.
\]
\end{proof}
\begin{proposition}
\label{prop:counterexample2}
Let $M^6=Q^6\# (K^4\times S^2)$ be as in
Proposition \ref{prop:counterexample1} and
let $L$ be the line bundle on $M$ with 
$c_1(L)=7h-14z$.  Then $M$ is simply connected
and $(M,L)$ is spin$^c$ bordant to a manifold
of generalized positive scalar curvature, but
$(M,L)$ does \textbf{not} admit
generalized positive scalar curvature.  Thus
Theorem \ref{thm:bordismspinc} would be false without the
hypothesis that $c_1(L)\co H_2(M)\to\bZ$ is surjective.
\end{proposition}
\begin{proof}
That $M$ is simply connected follows from
simple connectivity of
complex hypersurfaces in $\CP^n$, $n\ge 3$.
(This comes from the Lefschetz Hyperplane Theorem;
see for example \cite[p.\ 159]{MR1288523}.)
To see that $(M,L)$ is spin$^c$ bordant to a manifold
of generalized positive scalar curvature, 
observe that $\Omega_6^{\spinc}\cong \bZ^2$, and
that the bordism class of any spin$^c$ $6$-manifold
$(M,L)$ is detected by $\langle c_1(L)^3, [M]\rangle$
and by $\langle c_1(L)p_1(M), [M]\rangle$.
From the formulas in Proposition
\ref{prop:counterexample1}, we see that
$7h-14z$ is an allowable spin$^c$ line bundle on
$M$, and since $p_1=-24\Ahat$
in degree $4$ cohomology, spin$^c$ bordism
classes of $6$-manifolds are determined by
$\langle c_1(L)^3, [M]\rangle$ and
by $\alpha^c(M,L)$.  Now 
$\alpha^c(Q,h)=0$ and $\langle h^3, [Q]\rangle=2$
(since $Q$ has degree $2$ in $\CP^4$).  On the other
hand, for the given choice of $L$ with $c_1(L)=7h-14z$, 
$c_1(L)^3=7^3h^3$, so $\langle c_1(L)^3, [M]\rangle
=2\cdot 7^3=686$ and
\[
\alpha^c(M,L) = \frac{7\cdot (7^2-1)}{24}-14 = 0.
\]
That shows that $(M,L)$ is spin$^c$ bordant to
a disjoint sum of $343$ copies of $(Q,h)$, which
admits generalized positive scalar curvature
by Corollary \ref{cor:nonspinpsc} and Theorem
\ref{thm:bordismspinc}: note that $Q$ is simply connected and not
spin, and that $h$ is primitive, since $H_2(Q)\cong\bZ$ is generated
by the class of a line.  (Alternatively,
in the notation of Section \ref{sec:CP2}
below, $[(Q,h)]= T_{\MSpinc}(\CP^1,\cO,
\cO^2\oplus\cO(1))$, and such manifolds admit generalized
positive scalar curvature by Lemma \ref{lem:T}(d).)
On the other hand, $L=E^{\otimes 7}$ with
$c_1(E)=h-2z$, and again by Proposition
\ref{prop:counterexample1}, we have
$\alpha^c(M,E)= 0 - 2 = -2\ne 0$, so $(M,E)$ does not admit
generalized positive scalar curvature, and we can
apply Lemma \ref{lem:obstr} to get the
conclusion.
\end{proof}
\begin{remark}
\label{rem:AGAG}
In \cite[Theorem 1.2]{MR4817674}, it was asserted that vanishing
of $\alpha^c(M,L)$ is necessary and sufficient for a 
non-spin simply connected spin$^c$ manifold to admit
generalized positive scalar curvature.  Proposition
\ref{prop:counterexample2} shows that 
this assertion is incorrect as stated.  However, by Theorem
\ref{thm:bordismspinc} and Corollary \ref{cor:nonspinpsc}, it is
true in the case where the classifying map
$M\to\BSpinc$ is a $2$-equivalence, i.e., when $c_1(L)$ is primitive.
This is all that's
needed for the example in \cite[\S 6]{MR4817674}.
\end{remark}

\section{\texorpdfstring{$\bC\bP^2$ bundles}{CP2-bundles}}
\label{sec:CP2}

In this section we prove some results about spin$^c$ bordism and
$\bC\bP^2$ bundles needed for the proof of Theorem \ref{thm:S1psc}.
These might be of independent interest.
(Indeed, they were hinted at in \cite[p.\ 235]{MR1251581}, but the work
of R.\ Jung alluded to there never appeared.)  Related calculations
on $\bC\bP^2$-bundles and oriented bordism can be found in the
Diplom thesis of F{\"u}hring \cite{Fuhring}.

First we recall some basic facts about
spin$^c$ bordism.  The basic references are \cite[Chapter XI]{MR0248858},
\cite[\S8]{MR0234475}, and \cite{MR1166518}.  There is an
equivalence $\MSpinc\simeq \Sigma^{-2}\MSpin\wedge \bC\bP^\infty$, corresponding
to the isomorphism of bordism groups
$\Omega^{\spinc}_n\cong
\widetilde\Omega^{\spin}_{n+2}(\bC\bP^\infty)$;
see the proof of Lemma \ref{lem:S1cofib}. Furthermore, classes
in Spin$^c$ bordism are detected by their Stiefel Whitney numbers
(which are constrained just by the Wu relations and the vanishing of
$w_1$ and $w_3$) and integral cohomology characteristic numbers
(where in addition to the Pontryagin classes, one can use powers of $c_1$
of the line bundle defining the spin$^c$ structure)
\cite[Theorem, p.\ 337]{MR0248858}. \emph{Note that
the bordism class can change}, depending on the choice of spin$^c$
structure. Thus, for example, $\Omega^{\spinc}_2\cong \bZ$, with all classes
represented by $(\bC\bP^1, L)$, $L$ a complex line bundle with $c_1(L)$
even, and the isomorphism to $\bZ$ is given by
$(\bC\bP^1, L)\mapsto \langle c_1(L),\,[\bC\bP^1]\rangle/2$.
Similarly, $\Omega^{\spinc}_4\cong \bZ^2$, with one generator
given by $(\bC\bP^1, \cO(2))^2$, on which
$\alpha^c$ takes the value $1$, and the other generator given
by $(\bC\bP^2, \cO(1))$, where $c_1$ of the anticanonical
bundle $\cO(1)$ is the standard generator $x$ of $H^2(\bC\bP^2;\bZ)$, 
on which $\alpha^c$ takes the value $0$.  The calculation of $\alpha^c$
on this generator is worked out by Hattori \cite{MR508087}
(we thank the referee for this reference):
\[
\begin{aligned}
\alpha^c(\bC\bP^2,\cO(1))&=\ind\Dirac_{\bC\bP^2,\cO(1)} \\
&= \langle \widehat {\mathcal A}(\bC\bP^2)e^{x/2}, [\bC\bP^2]\rangle\\
&= \big\langle (1-\tfrac{1}{8}x^2)
(1 + \tfrac{1}{2}x + \tfrac{1}{2}\tfrac{x^2}{4}),
[\bC\bP^2]\big\rangle = 0,
\end{aligned}
\]
by the Atiyah-Singer Theorem \cite[Theorem D.15, p.\ 399]{lawson89:_spin}.

This last example turns out to be crucial, because there is a sense
in which $\bC\bP^2$ with the bundle $\cO(1)$, the dual of 
the tautological  bundle, generates the kernel of $\alpha^c$.
In fact $\bC\bP^2$ plays the same role for $\alpha^c$ that $\bH\bP^2$
plays for $\alpha$ according to the work of Stolz.

We now use $(\bC\bP^2,\cO(1))$ to construct a transfer map
$T_{\MSpinc}\co \Omega^{\spinc}_n(BG)\to \Omega^{\spinc}_{n+4}$, where
$G$ is the Lie group $\U(3)$, as follows.  The group $\U(3)$ acts
transitively on $\bC\bP^2\cong G/H$, where $H=\U(2)\times \U(1)$,
preserving the bundle $\cO(1)$ (with its Fubini-Study connection).
We notice that a map $f\co B\to BG=B\U(3)$ from a
closed manifold $B$ is the same thing as a complex vector bundle $V\to
B$ of rank $3$, and the associated $\bC\bP^2$ bundle $B\times_f
\bC\bP^2$, obtained by pulling back (under $f$) the $G/H$-bundle over
$BG$ associated to the universal principal $G$-bundle, is the bundle
$\pi\co \bP(V)\to B$ of lines in $V$.  Let $\xi=c_1(\cO_{\bP(V)}(1))$.
Given a closed spin$^c$ manifold $(B, L_B)$ and $V$ as above, we put
\begin{equation}
\label{eq:Tdef}
T(B,L_B,V):=\Big(\bP(V),\ 
\cO_{\bP(V)}(1)\otimes\pi^*(L_B\otimes\det V)\Big),
\end{equation}
with the spin$^c$ structure described in Lemma \ref{lem:T}(a) below.
This manifold
has dimension $n+4$ if $\dim B=n$, and its spin$^c$ structure is
inherited from
the spin$^c$ structure on $B$ defined by $L_B$, the spin$^c$ structure
on $\bC\bP^2$ defined by $\cO(1)$, and the complex structure on $V$.
By Lemma \ref{lem:T}(a), $T$ defines
a homomorphism
\[
T_{\MSpinc}:=T\co\Omega^{\spinc}_n(B\U(3))\to\Omega^{\spinc}_{n+4}.
\]
(In the published version of this paper, $T_{\MSpinc}$ was defined
using $\SU(3)$ in place of $\U(3)$, i.e., only for $V$ with trivial
determinant; see Remark \ref{rem:SU3} below for why this is not
good enough.)

We will also need another transfer map introduced in \cite{MR1189863}
and \cite{MR1259520}.  This is defined similarly, but with $G=\SU(3)$
replaced by $\PSp(3)$, $H=S(\U(2)\times \U(1))$ replaced by
$P(\Sp(2)\times \Sp(1))$, and $\bC\bP^2$ replaced by $\bH\bP^2$.  One
obtains a transfer map $T_{\MSpin}\co \Omega^{\spin}_n(B\PSp(3))\to
\Omega^{\spin}_{n+8}$ which can be extended to a similar map on any
$\MSpin$-module spectrum.  We will apply it to the
spectrum $\MSpinc\simeq \MSpin\wedge \Sigma^{-2}\bC\bP^\infty$.

We're now ready for one of the main theorems of this paper.
\begin{theorem}
\label{thm:transfer}  
The transfer maps
\[
T_{\MSpinc}\co \Omega^{\spinc}_n(B\U(3))\to \Omega^{\spinc}_{n+4}
\ \ \mbox{and} \ \
T_{\MSpin}\co \Omega^{\spinc}_n(B\PSp(3))\to \Omega^{\spinc}_{n+8}
\]
defined above have images lying in the kernel of the
index map $\alpha^c\co \Omega^{\spinc}_*\to KU_*$. Furthermore, 
the kernel of $\alpha^c$ is
additively generated by the images of $T_{\MSpinc}$ and of
$T_{\MSpin}$.
\end{theorem}
The proof of Theorem \ref{thm:transfer} will be given at the end of this
section, after some preliminaries (Lemmas \ref{lem:T}, \ref{lem:Z},
and \ref{lem:gcd}).  First we record the consequence that is needed
for the proof of Theorem \ref{thm:S1psc}.
\begin{corollary}
\label{cor:nonspinpsc}
Let $(M^n,L)$ be a simply connected spin$^c$ manifold, $L$ the
line bundle over $M$ defining the spin$^c$ structure,
with $\alpha^c([M,L])=0$.
Then after changing $(M,L)$ up to spin$^c$ cobordism, we can assume that
$M$ admits a Riemannian metric $g$ of {\psc} and
the line bundle $L$ over $M$ defining the spin$^c$ structure
admits a hermitian metric $h$ with
$\frac14 \kappa_g + \cR_h > 0$.
\end{corollary}
\begin{proof}
As we have indicated, the idea of the proof is to reduce things to
the case of $(\bC\bP^2,\cO(1))$, so the first step is to prove the
theorem in this case.  For this case, no cobordism is necessary;
we use the Fubini-Study metric along with the usual connection on
the dual of the tautological bundle.  Then if $\omega$ is the
K\"ahler form, this is also the curvature of $\cO(1)$ and
the Ricci tensor is $6$ times the metric.  So $\frac14 \kappa$ is
$6$ while the minimal eigenvalue of $\cR$ is $-2$, so
$\frac14 \kappa_g + \cR_h \ge 6-2 = 4 > 0$.

Now we deal with the general case, using Theorem \ref{thm:transfer}.
Since the kernel of $\alpha^c$ is additively generated by the images
of the two transfer maps in Theorem \ref{thm:transfer}, and since a
disjoint union of spin$^c$ manifolds (with either orientation)
carrying metrics and connections with
$\frac14 \kappa + \cR > 0$ again carries such metrics and connections,
we can deal separately
with the images of the two different transfer maps that appear in
Theorem \ref{thm:transfer}.  If $M$ is in the image of
$T_{\MSpin}$ and is the total space of an $\bH\bP^2$-bundle
with structure group $\PSp(3)$, we can rescale the $\bH\bP^2$ fiber
to have very small diameter and big curvature, and then the
scalar curvature of the $\bH\bP^2$ fiber will dominate everything
else.  So in that case the conclusion is obvious.  Thus we
can assume after making a spin$^c$ cobordism that $M^n$ is the
total space of a $\bC\bP^2$-bundle $\bP(V)\to B$ with structure
group $\U(3)$, with $L_M$ restricting to
$\cO(1)$ on the fibers.  Once again, by choosing the metric
and connection so that on each fiber, we have a very small multiple
of the Fubini-Study metric and the curvature of the line bundle is
the K\"ahler form, the curvature of the fibers will dominate everything
else, and the conclusion follows.  (See Lemma \ref{lem:T}(d) below
for more details.)
\end{proof}  

We now establish the basic properties of the transfer map
$T_{\MSpinc}$ defined in \eqref{eq:Tdef}.
\begin{lemma}\label{lem:T}
Let $(B,L_B)$ be a closed spin$^c$ manifold and $V\to B$ a complex
vector bundle of rank $3$, as in \eqref{eq:Tdef}, and let
$E=\bP(V)$, $\pi\co E\to B$, and $T_v$ the vertical tangent bundle.
\begin{enumerate}
\item[{\rm (a)}] There are isomorphisms
  \[
  TE\oplus\bC\cong\pi^*TB\oplus\pi^*V\otimes\cO(1),
  \ \ \   T_v\oplus\bC\cong\pi^*V\otimes\cO(1),
  \]
  and $c_1(T_v)=3\xi+\pi^*c_1(V)$. The spin$^c$ structure on $\pi^*TB$
  with determinant $\pi^*L_B$, the complex structure on
  $\pi^*V\otimes\cO(1)$, and the twist by the line bundle $\cO(-1)$
  define a spin$^c$ structure on $E$ with determinant
  \[
  L_E:=\cO(1)\otimes\pi^*(L_B\otimes\det V).
  \]
  It depends only on the
  bordism class of $(B,L_B,V)$. Thus $T_{\MSpinc}$
  is well defined.
\item[{\rm (b)}] The image $\im T_{\MSpinc}$ is an ideal of
  $\Omega^{\spinc}_*$.
\item[{\rm (c)}] We have
  $\alpha^c(T_{\MSpinc}(B,L_B,V))=0$. More generally, for $n\in\bZ$,
\begin{equation}\label{eq:GT}
\ind\big(D_{L_E}\otimes L_E^{\otimes n}\big)=\Big\langle \Ahat(B)\,e^{(2n+1)c_1(L_B)/2}\,e^{n\,c_1(V)}\,\ch\big(\pi_!\cO(n-1)\big),[B]\Big\rangle,
\end{equation}
where $\pi_!\cO(j)=\Sym^jV^*$ for $j\ge0$, $=0$ for $j=-1,-2$, and $=\Sym^{-3-j}V\otimes\det V$ for $j\le-3$.
\item[{\rm (d)}] The manifold $T_{\MSpinc}(B,L_B,V)$
  admits generalized positive scalar curvature, i.e.,
  the total space $E$ admits a Riemannian metric and
  $L_E$ a hermitian metric and unitary connection with
  $\frac14\kappa_E+\cR_{L_E}>0$.
\end{enumerate}
\end{lemma}

\begin{proof}
  (a) The splitting of $TE\oplus\bC$ is standard; $w_2(E)\equiv
  w_2(B)+c_1(V)+3\xi\equiv c_1(L_E)\pmod2$, and twisting a spin$^c$
  structure by a line bundle $N$ multiplies its determinant by
  $N^{\otimes2}$, which gives the claim about $L_E$.

  (b) We have $(N,L_N)\times
    T_{\MSpinc}(B,L_B,V)=T_{\MSpinc}(N\times B,L_N\boxtimes
    L_B,\mathrm{pr}_B^*V)$.

  (c) Since $\Ahat(T_v)e^{c_1(T_v)/2}=\Td(T_v)$, we have:
\[
\Ahat(E)\,e^{(2n+1)c_1(L_E)/2}=\Ahat(B)\,e^{(2n+1)c_1(L_B)/2}\,e^{n c_1(V)}\cdot\Td(T_v)\,e^{(n-1)\xi},
\]
and $\pi_*\big(\Td(T_v)e^{(n-1)\xi}\big)=\ch(\pi_!\cO(n-1))$ by the
Grothendieck-Riemann-Roch theorem for the projective bundle, and
we get the cohomology of $\pi_!\cO(j)$ from the
cohomology of  $\cO(j)$ on $\CP^2$ together with relative Serre duality
($K_{E/B}=\cO(-3)\otimes\pi^*\det V^*$). For $n=0$ we get
$\pi_!\cO(-1)=0$.

(d) The argument goes exactly as in the proof of
Corollary \ref{cor:nonspinpsc}: the group $\U(3)$
acts isometrically on $(\CP^2,g_{FS})$ and preserves $\cO(1)$ with its
Fubini--Study connection, so $E$ carries submersion metrics with
fibers $\epsilon^2g_{FS}$ and $L_E$ carries the tensor product of the
fiberwise Fubini-Study connection with any connection on
$\pi^*(L_B\otimes\det V)$; then $
\frac14\kappa_E+\cR_{L_E}\ge4\epsilon^{-2}-C>0 $ for $\epsilon$ small
enough.
\end{proof}
\begin{remark}
\label{rem:SU3}
The reason for using $\U(3)$ rather than $\SU(3)$ in the definition of
$T_{\MSpinc}$ is to have $\ker\alpha^c\otimes\bZ_{(p)}= \im
T_{\MSpinc}\otimes\bZ_{(p)}$ for every odd prime $p$, as required by
Theorem \ref{thm:transfer}.  This is not true for $p=3$ if one only
allows $\SU(3)$-bundles, since $\pi_1(\mathrm{PU}(3))\cong \bZ/3$, and
with that definition, the image of the transfer map in degree $6$ is
generated by the product of spin$^c$ manifolds  $
(\CP^1,\cO(2))\times (\CP^2,\cO(1)), $ which has characteristic
numbers $(c_1^3,p_1c_1)=(6,6)$, whereas the quadric complex $3$-fold
$(Q,h)$ of Proposition \ref{prop:counterexample1} also lies in the
kernel of $\alpha^c$ and has characteristic numbers
$(c_1^3,p_1c_1)=(2,2)$.
\end{remark}

For the proof of Theorem \ref{thm:transfer} at odd primes we will need
explicit generators of complex bordism lying in the image of
transfer map $T_{\MSpinc}$.  Recall that for a
stably almost complex manifold $M^{2k}$ of complex dimension $k$,
$s_k(M)=\langle s_k(c_1,\dots,c_k)(TM),[M]\rangle$ denotes the Milnor
number of $M$, i.e., the characteristic number corresponding to the
symmetric polynomial $\sum x_i^k$.
\begin{definition}
\label{def:Zmanifolds}  
Let $B$ be a closed stably almost complex manifold and $N_1,N_2$
complex line bundles on $B$, $n_i=c_1(N_i)$. Put $V=\cO\oplus
N_1\oplus N_2$ and $E=\bP(V)$. By Lemma \ref{lem:T}(a), we have the
isomorphism
\[
TE\oplus\bC\cong\pi^*TB\oplus\cO(1)\oplus\cO(1)N_1\oplus\cO(1)N_2 .
\]
Let $Z=Z(B;N_1,N_2)$ denote the total space $E$ of the
  bundle $E\to B$ equipped with the stable complex structure obtained
by replacing the summand $\cO(1)$ by its conjugate $\cO(-1)$. Its
underlying oriented manifold is $-E$, and $c_1(Z)=c_1(B)+\xi+n_1+n_2$.
\end{definition}
Let $\iota\co\MU\to\MSpinc$ be the natural map (assigning to a
stably almost complex manifold the spin$^c$ structure given by
the anti-canonical bundle $K_B^{-1}=\det TB$); it is a ring map of
ring spectra, and the composition
$\Td=\alpha^c\circ\iota$ is the Todd genus. The map $\iota$ is
surjective after inverting $2$ \cite[pp. 348--349]{MR0248858}.
\begin{lemma}\label{lem:Z}\leavevmode
\begin{enumerate}
\item[{\rm (a)}] We have $\iota(Z(B;N_1,N_2))=-T(B,K_B^{-1},\cO\oplus
  N_1\oplus N_2)$ in $\Omega^{\spinc}_*$ whenever $E$ is simply
  connected; in particular $\Td(Z)=0$.
\item[{\rm (b)}] Let $\dim_{\bC}B=k-2$, $k\ge2$. Then $s_k(Z)=-\langle\Phi_k(n_1,n_2),[B]\rangle$ where, with $\pi_*\xi^{j+2}=(-1)^jh_j(n_1,n_2)$ {\lp}$h_j$ the complete symmetric polynomial{\rp},
\[
\begin{array}{rcl}
  \Phi_k(n_1,n_2)&=&\displaystyle
  \pi_*\big((\xi+n_1)^k+(\xi+n_2)^k+(-\xi)^k\big)=\sum_{i=0}^{k-2}\phi_i\,n_1^in_2^{k-2-i},
  \\
  \phi_i&=&\displaystyle
  1+(-1)^{k-i}\tbinom{k-1}{i}+(-1)^i\tbinom{k-1}{i+1}.
\end{array}
  \]
\item[{\rm (c)}] For
  $Z_i:=Z(\CP^i\times\CP^{k-2-i};\cO(1,0),\cO(0,1))$, $0\le i\le k-2$,
  one has $s_k(Z_i)=-\phi_i$.
\end{enumerate}
\end{lemma}

\begin{proof}
(a) The manifold $\iota(Z)$ is $-E$ with the spin$^c$
  structure given by the determinant
  \[
  \cO(1)\otimes\pi^*(K_B^{-1}\otimes
  N_1\otimes N_2),
  \]
  which is $L_E$ for $T(B,K_B^{-1},V)$.  Recall that
  on a simply connected manifold, a spin$^c$ structure is determined by
  its determinant and the orientation, so $\iota(Z)$ is the bordism
  inverse of $T(B,K_B^{-1},V)$. Then $\Td(Z)=\alpha^c(\iota Z)=0$ by
  Lemma \ref{lem:T}(c).  One can also check directly that
\begin{multline*}
\pi_*\big(\Td(\cO(-1))\Td(\cO(1)N_1)\Td(\cO(1)N_2)\big)\\=G(n_1)G(n_2)-G(n_1)G(n_2-n_1)-G(n_2)G(n_1-n_2)=0
\text{ for }G(t)=(1-e^{-t})^{-1}.
\end{multline*}

(b) The Chern classes of $Z$  are the same
  as those of $\pi^*TB$ (which contribute nothing to $s_k$ for dimensional
  reasons) together the classes of $-\xi$, $\xi+n_1$, $\xi+n_2$; and
  $[Z]=-[E]$.
The relation
$\xi(\xi+n_1)(\xi+n_2)=0$ with $\pi_*\xi^2=1$ gives
$\pi_*\xi^{j+2}=(-1)^jh_j(n_1,n_2)$. Expanding, we obtain
\[
\Phi_k(n_1,n_2)=\sum_{m=2}^k\tbinom
km(-1)^mh_{m-2}(n_1,n_2)\big(n_1^{k-m}+n_2^{k-m}\big)+h_{k-2}(n_1,n_2),
\]
and reading off the coefficient of $n_1^in_2^{k-2-i}$ with
$\sum_{m=j}^k(-1)^m\binom km=(-1)^j\binom{k-1}{j-1}$ gives 
the expression for the coefficient $\phi_i$.

(c) is immediate from (b). For instance $\Phi_3=2(n_1+n_2)$ and
$\Phi_4=5(n_1^2-n_1n_2+n_2^2)$.
\end{proof}

Recall \cite[Ch.\ VII]{MR0248858} that
$\Omega^U_*\otimes \bZ_{(p)}\cong
  \bZ_{(p)}[z_2,z_4,\ldots]$ with the polynomial generators $z_{2j}$
given by almost complex manifolds of (complex) dimensions
$j=1,2,3,\cdots$, and that an almost complex manifold $M$ of complex
dimension $j$ can be chosen to be a polynomial generator of
$\Omega^U_{2j}\otimes \bZ_{(p)}$ for a prime $p$ exactly when
$v_p(s_j(M))=v_p(m_j)$, where $v_p$ is the $p$-adic valuation and
$m_j=p$ if $j+1$ is a power of $p$ and $m_j=1$ otherwise.  The number
$m_j$ always divides $s_j(x)$ for every $x\in \Omega^U_{2j}$.
\begin{lemma}\label{lem:gcd}
Let $k\ge2$ and let $p$ be an odd prime. Then 
$\displaystyle\min_{0\le i\le
  k-2}v_p(\phi_i)=v_p(m_k)$.
\end{lemma}

\begin{proof}
It suffices to show $\min_iv_p(\phi_i)\le v_p(m_k)$, 
as the reverse inequality comes for free.

$\bullet$ \emph{$k$ even.} Then $\phi_i=1+(-1)^i\binom{k}{i+1}$ and
  $m_k=p$ iff $k+1=p^j$. If $p\mid\phi_i$ for all $i$, then $\binom
  km\equiv(-1)^m\pmod p$ for $1\le m\le k-1$, hence
  $\binom{k+1}{m}\equiv0$ for $1\le m\le k$,
  i.e.\ $(1+x)^{k+1}\equiv1+x^{k+1}\pmod p$, which forces
  $k+1=p^j$. So if $k+1$ is not a power of $p$ some $\phi_i$ is a
  $p$-adic unit. If $k+1=p$, then $\phi_0=1+\binom{p-1}{1}=p$. If
  $k+1=p^j$ with $j\ge2$ and $p^2\mid\phi_i$ for all $i$, the same
  computation gives $\binom{p^j}{m}\equiv0\pmod{p^2}$ for $2\le m\le
  p^j-2$; but $v_p\left(\binom{p^j}{p^{j-1}}\right)=1$ by 
  Kummer's theorem, a contradiction since $2\le p^{j-1}\le p^j-2$. 
  Hence some $\phi_i$ has $v_p(\phi_i)=1$.

$\bullet$ \emph{$k$ odd.} Then $k\ge3$, $m_k=1$, and
  $\phi_i=1-(-1)^i\binom{k-1}{i}+(-1)^i\binom{k-1}{i+1}$. If
  $p\mid\phi_i$ for all $i$, the coefficients of $(1+x)^{k-1}(1-x)$ in
  degrees $1,\dots,k-1$ are $\equiv(-1)^m\pmod p$, so
  \[
  (1+x)^{k-1}(1-x)\equiv\sum_{m=0}^k(-1)^mx^m=(1-x^{k+1})/(1+x),
  \]
  i.e.\ $(1+x)^k(1-x)\equiv1-x^{k+1}\pmod p$. Comparing the
  coefficients of $x$ and $x^2$ gives $k\equiv1$ and
  $k(k-3)/2\equiv0$, i.e.\ $-1\equiv0\pmod p$, a contradiction.
\end{proof}

We now turn to the preliminaries needed for the proof of Theorem
\ref{thm:transfer} at the prime $2$.  The argument is a version of the
bottom-cell construction of Hopkins and Hovey
\cite[\S\S4--5]{MR1166518}: after collapsing out the image of the
$\bH\bP^2$-transfer, $\MSpinc$ is additively a wedge of suspensions of
$\ku$, and explicit $\bC\bP^2$-bundles over quaternionic projective
spaces detect the successive bottom classes modulo $2$.

For $j\ge1$, let $\gamma$ be the tautological quaternionic line bundle
over $\HP^{j-1}$, regarded as a complex rank-two $\SU(2)$-bundle.  Set
\begin{equation}\label{eq:ej}
E_j=\bP(\gamma\oplus\bC),\qquad
e_j=[E_j,\cO(1)]\in\Omega_{4j}^{\spinc}.
\end{equation}
The base $\HP^{j-1}$ is spin and $\det(\gamma\oplus\bC)$ is trivial, so
$e_j=T_{\MSpinc}(\HP^{j-1},\bC,\gamma\oplus\bC)$ in the notation of
\eqref{eq:Tdef}; in other words, the $e_j$ lie in the image of
$T_{\MSpinc}$, via $\SU(2)\subset\SU(3)\subset\U(3)$.  In particular,
$e_1=[\CP^2,\cO(1)]$.  By Lemma \ref{lem:T}, $\alpha^c(e_j)=0$, and
$E_j$ with the line bundle $\cO(1)$ admits generalized positive scalar
curvature.
\begin{lemma}\label{lem:geometry}
Let $\xi=c_1(\cO(1))$ on $E_j$, $j\ge 1$.  Then
\[
\langle\xi^{2j},[E_j]\rangle=\pm1.
\]
\end{lemma}
\begin{proof}
Choose
$u\in H^4(\HP^{j-1};\bZ)$ with $c_2(\gamma)=\pm u$. The projective-bundle
relation and fibre integration $\pi_*$ for $\pi\co E_j\to \HP^{j-1}$ give
\[
\xi^3+\pi^*c_2(\gamma)\xi=0,\qquad \pi_*(\xi^2)=1.
\]
Consequently $ \pi_*(\xi^{2j})=(-c_2(\gamma))^{j-1}.  $ Its evaluation
on $\HP^{j-1}$ is $\pm1$, also when $j=1$.
\end{proof}

From now on until the proof of Theorem \ref{thm:transfer}, all spectra
and homotopy groups are localized at $2$, and cohomology without
coefficients specified means mod-$2$ cohomology.  The symbol
$\CP^\infty$ denotes the based space, used as a reduced suspension
spectrum.  As in the proof of Lemma \ref{lem:S1cofib}, we use the
standard equivalence of $\MSpin$-modules
$\MSpinc\simeq\MSpin\wedge\Sigma^{-2}\CP^\infty$,
and define
\begin{equation}\label{eq:q}
q\co\MSpinc\longrightarrow \cC:=\ko\wedge\Sigma^{-2}\CP^\infty
\end{equation}
by the spin orientation $\MSpin\to \ko$. Since $\MSpinc$ is
an $\MSpin$-module spectrum, we can apply the main results
(Theorem B and Proposition 8.3) of \cite{MR1259520},
applied to $\Sigma^{-2}\CP^\infty$.  They
give
\begin{equation}\label{eq:quotient}
q_*\co\Omega_*^{\spinc}/\im T_{\MSpin}\xrightarrow{\;\cong\;}\pi_*\cC .
\end{equation}
The transfer $T_{\MSpin}$ here, defined via the
$\MSpin$-module structure, agrees with the geometric
$\bH\bP^2$-transfer defined above by the
$\MSpin$-module identification. Also, by \cite[Theorem~2.1]{MR3120631},
there is an additive equivalence
\begin{equation}
\label{eq:kowedgeCP}  
\cC=\ko\wedge \Sigma^{-2}\bC\bP^\infty\simeq \bigvee_{k=0}^\infty \Sigma^{4k}\ku.
\end{equation}
Only this additive equivalence will be used.
(The summands on the right are some of the suspensions
$\Sigma^{4n(J)}\ku$ of $\ku$, indexed by partitions $J$
of size $n(J)$, that appear in the Anderson-Brown-Peterson additive splitting
of $\MSpinc$ localized at $2$.  This splitting is implicit in
\cite{MR0219077} and written out in detail in \cite[Ch.\ XI]{MR0248858},
\cite[\S8]{MR0234475}, and in \cite{MR1166518}.  The basic result is
that, localized at $2$, $\MSpinc$ splits additively as a direct sum
of suspensions $\Sigma^{4n(J)}\ku$ of $\ku$, indexed by partitions $J$
of size $n(J)$, together with some suspensions of $H\bZ/2$ that start in
fairly high dimension.  The lowest-dimensional summand is $\ku$ itself,
coming from the empty partition $\emptyset$ with $n(\emptyset)=0$,
and $\alpha^c$ is just projection onto this bottom summand.
It may also be helpful to explain where the splitting
\eqref{eq:kowedgeCP} comes from.  It is well known (see
\cite[Corollary 5.5]{MR1189863}, \cite[Proposition 2.3]{MR1259520},
\cite[Theorem 1.6.1]{Bruner}, and
\cite[Theorems 3.1.17 and 3.1.25]{MR860042})
that $H^*\MSpin$, $H^*\ko$, $H^*\ku$, and
$H^*$ of $\MSpin$-module spectra which are bounded below and
locally finite are all, as modules over the mod-$2$ Steenrod algebra,
extended modules induced from the finite-dimensional subalgebra
$A(1)$ generated by $1=\Sq^0$, $\Sq^1$ and $\Sq^2$.  The Adams
spectral sequences of all these spectra collapse, and so the
splitting \eqref{eq:kowedgeCP} comes from the splitting of
$\Sigma^{-2}\bC\bP^\infty$ as an $A(1)$-module into
$\bigvee_{k=0}^\infty \Sigma^{4k}C$, where $C$ is the very simple
$2$-dimensional $A(1)$-module symbolized by the diagram
$\xymatrix@C-2ex{\bullet_0\ar@/^/[rr] && \bullet_2}$, where the
curved arrow represents the action of $\Sq^2$.  And it is
classical (due to Wood and Fujii) that $\ko\wedge C\simeq \ku$.)

The Anderson--Brown--Peterson splitting provides a map of spectra
\begin{equation}\label{eq:section}
s\co\ku\longrightarrow \MSpinc,\qquad \alpha^c\circ s\simeq\mathrm{id}_{\ku};
\end{equation}
see \cite[\S2]{MR1166518} and \cite[Chapter~XI]{MR0248858}.
Recall that $\alpha^c$ is a ring spectrum map; however,
no multiplicativity of $s$ is asserted.

Let $h$ denote the mod-$2$ stable Hurewicz map.

\begin{lemma}\label{lem:ku}
The following properties hold.
\begin{enumerate}
\item Every positive-degree element of $\pi_*\ku$ has zero Hurewicz image.
\item If $a>b$, a map $\Sigma^a \ku\to\Sigma^b \ku$ induces zero in
mod-$2$ cohomology.
\item A self-map of $\Sigma^a \ku$ that is odd on $\pi_a$ induces the
identity in mod-$2$ cohomology.
\end{enumerate}
\end{lemma}

\begin{proof}
The standard description
\[
H^*(\ku;\bF_2)\cong\cA\otimes_{E(1)}\bF_2
\]
shows that this Steenrod module is generated by its degree-zero class;
here $\cA$ is the mod-$2$ Steenrod algebra and $E(1)$ is the exterior
subalgebra on $Q_0=\Sq^1$ and
$Q_1=\Sq^1\Sq^2+\Sq^2\Sq^1$.
The generator must map to zero under the cohomology map induced by
Bott multiplication $\Sigma^2\ku\to \ku$. Thus Bott multiplication is
zero in mod-$2$ homology. Since
$\pi_*\ku=\bZ_{(2)}[\beta]$, $|\beta|=2$, this proves (1).
The same degree argument proves (2). For (3), oddness on the bottom
homotopy group fixes the bottom cohomology generator, hence the entire
Steenrod module.
\end{proof}

\begin{lemma}\label{lem:detect}
For $j\ge1$, $h(q_*e_j)\ne0$. Consequently, under any fixed
equivalence \eqref{eq:kowedgeCP}, the component of $q_*e_j$ in
$ \pi_{4j}(\Sigma^{4j}\ku)\cong\bZ_{(2)} $ is odd.
\end{lemma}

\begin{proof}
Let $x\in H^2(\CP^\infty;\bF_2)$ be the standard generator.
The bottom orientation $\ko\to H\bF_2$ gives
\[
\cC\longrightarrow H\bF_2\wedge\Sigma^{-2}\CP^\infty.
\]
Composing with the functional represented by $x^{2j+1}$ defines
$\lambda_j\in H^{4j}(\cC;\bF_2)$. The Thom identification gives, for a
$4j$-dimensional spin$^c$ manifold $(W,L)$,
\begin{equation}\label{eq:pairing}
\langle\lambda_j,h(q_*[W,L])\rangle
=\langle c_1(L)^{2j},[W]\rangle\pmod2.
\end{equation}
For completeness, geometrically the equivalence
$\Omega_{4j}^{\spinc}\cong\widetilde\Omega_{4j+2}^{\spin}(\CP^\infty)$
is given in the reverse direction by taking the zero set of a transverse
section of a line bundle $\widetilde L$ on a spin manifold $N$.
The resulting determinant is $\widetilde L|_W$, up to the harmless
choice of dual convention modulo $2$, and
\[
\langle c_1(\widetilde L|_W)^{2j},[W]\rangle
=\langle c_1(\widetilde L)^{2j+1},[N]\rangle\pmod2.
\]
The composite $\MSpin\to \ko\to H\bF_2$ is the ordinary mod-$2$ Thom
orientation. This proves \eqref{eq:pairing}.

Lemma~\ref{lem:geometry} now gives
$\langle\lambda_j,h(q_*e_j)\rangle=1$.
In degree $4j$, the summands $\Sigma^{4i}\ku$ with $i>j$ have no
homotopy. Those with $i<j$ contribute positive-degree coefficients of
$\ku$, whose Hurewicz images vanish by Lemma~\ref{lem:ku}(1).
Thus the nonzero Hurewicz image comes from the bottom of the
$j$th summand, and its coefficient is odd.
\end{proof}

Now let $\mu\co\MSpinc\wedge \MSpinc\to \MSpinc$ be the multiplication.
Put $e_0=1$ and define
\[
f_j\co\Sigma^{4j}\ku
\xrightarrow{\,e_j\wedge s\,}\MSpinc\wedge \MSpinc
\xrightarrow{\mu}\MSpinc\xrightarrow{\,q\,} \cC,\qquad j\ge0.
\]
Here $e_j$ also denotes a representing map $S^{4j}\to \MSpinc$; for $j=0$
the composite is homotopic to $qs$. Set
\begin{equation}\label{eq:Phi}
\Phi=\bigvee_{j\ge0}f_j\co
\bigvee_{j\ge0}\Sigma^{4j}\ku\longrightarrow \cC .
\end{equation}
This is the bottom-cell multiplication construction of
\cite[\S5, pp.~189--190]{MR1166518}, with the classes \eqref{eq:ej} supplying
the bottoms. In the present situation the quotient by the
$\bH\bP^2$-transfer in \eqref{eq:quotient}
allows the argument to stay connective.

\begin{proposition}\label{prop:Phi}
The map $\Phi$ is an equivalence of spectra.
\end{proposition}

\begin{proof}
Choose an additive equivalence \eqref{eq:kowedgeCP} and let
$p_i\co \cC\to\Sigma^{4i}\ku$ denote its projections. Write
$g_{ij}=p_if_j$. If $i<j$, then $g_{ij}^*=0$ by
Lemma~\ref{lem:ku}(2).

The section $s$ sends $1\in\pi_0\ku$ to the unit of $\pi_0\MSpinc$.
Thus $f_j$ sends its bottom generator to $q_*e_j$.
For $j\ge1$, Lemma~\ref{lem:detect} says that $g_{jj}$ is odd on
its bottom homotopy group. The same is true for $j=0$, since $q$
is an isomorphism on $\pi_0$. By Lemma~\ref{lem:ku}(3), each
$g_{jj}^*$ is the identity.

The matrix of $\Phi^*$ is therefore triangular with identity diagonal.
In each cohomological degree only finitely many summands occur, so
$\Phi^*$ is an isomorphism. The source and target are connective and
their homotopy groups are finitely generated over $\bZ_{(2)}$ in each
degree. A mod-$2$ homology equivalence between such spectra is an
equivalence: otherwise the first nonzero homotopy group of its
cofiber would have nonzero reduction modulo $2$, contradicting the
Hurewicz theorem. Hence $\Phi$ is an equivalence.
\end{proof}
Note that multiplication is used only in $\MSpinc$. The argument
requires neither a
ring-spectrum structure on $\cC$ nor a multiplicative or module-compatible
choice of the splitting \eqref{eq:kowedgeCP}. Off-diagonal maps are permitted;
their cohomological triangularity is enough.

We are now ready to prove Theorem \ref{thm:transfer}.
Recall that it is an integral statement; the material above, from
Lemma \ref{lem:ku} on, was $2$-local and will be used in the second
half of the proof.
\begin{proof}[Proof of Theorem \ref{thm:transfer}]
First of all let's observe that the image of the transfer $T_{\MSpinc}$ is
contained in the kernel of $\alpha^c$; this is Lemma \ref{lem:T}(c).
(Alternatively, the vertical spin$^c$ Dirac operator on a bundle in the
image of $T_{\MSpinc}$ is the fiberwise Dolbeault
operator twisted by $\cO(-1)$; since
$H^q(\CP^2;\cO(-1))=0$ for every $q$, its family index is zero,
and the composition formula for the index gives
$\alpha^c\circ T_{\MSpinc}=0$.)

The argument for the image of the transfer $T_{\MSpin}$ is similar.
If a bundle $\bH\bP^2\to M\xrightarrow{p} N$ is in the image of the
transfer, then since $\bH\bP^2$ is $3$-connected, the spin$^c$ line bundle
$L_M$ on $M$ is just the pull-back $p^*L_N$ of the spin$^c$ line bundle
$L_N$ on $N$.  Since $\widehat A(\bH\bP^2)=0$, we see from
Atiyah-Singer that $\alpha^c(M)=0$.  (Alternatively, the fiberwise
spin Dirac operators are invertible for the standard
metric of {\psc} on $\bH\bP^2$, which is preserved by $\PSp(3)$,
so their family index is zero.)

Now we have to prove that the two images generate.  Since $T_{\MSpinc}$ is the induced map on
homotopy groups of a map of spectra (which by abuse of notation we denote
by the same symbol) $T\co\MSpinc\wedge \Sigma^4B\U(3)_+\to \MSpinc$
(see \cite{Fuhring} for a very similar situation),
we can prove the desired result by localizing separately at and away
from the prime $2$.

\textbf{Proof at odd primes.} After tensoring with
$\bZ\left[\frac12\right]$, the natural map
$\iota \co\MU\to\MSpinc$ is surjective 
\cite[pp.\ 348--349]{MR0248858}, and is a ring
map of ring spectra, with
$\Td=\alpha^c\circ\iota$ the Todd genus.
Thus, after tensoring with $\bZ_{(p)}$ for $p$
an odd prime, $\ker\alpha^c$ is generated
by stably almost complex manifolds
with vanishing Todd genus. Now 
$\Omega^U_*\otimes \bZ_{(p)}\cong 
\bZ_{(p)}[z_2,z_4,\ldots]$ as recalled above, and we can take 
$z_2=[\CP^1,\cO(2)]$, which has Todd genus $1$.
Since the kernel of $\alpha^c$ is an ideal,
it suffices to show that we can construct 
the other polynomial generators $z_{2k}$, $k>1$,
to be in the image of $T_\MSpinc$ (as then
$\ker \Td\otimes\bZ_{(p)}$ is the ideal generated by the $z_{2k}$,
$k>1$, which $\iota$ maps into $\im T_{\MSpinc}\otimes\bZ_{(p)}$, since
$\im T_{\MSpinc}$ is an ideal by Lemma \ref{lem:T}(b)).
This is accomplished by Definition \ref{def:Zmanifolds}  
and Lemmas \ref{lem:Z} and \ref{lem:gcd}: for $k\ge 2$, the manifolds
$Z_i$, $0\le i\le k-2$, of Lemma \ref{lem:Z}(c) have complex dimension
$k$ and vanishing Todd genus, their images $\iota(Z_i)$ lie in
$\im T_{\MSpinc}$ by Lemma \ref{lem:Z}(a) (the total spaces being
simply connected), and $s_k(Z_i)=-\phi_i$, so by Lemma \ref{lem:gcd}
a suitable $\bZ_{(p)}$-linear combination $z_{2k}$ of the $Z_i$
satisfies $v_p(s_k(z_{2k}))=v_p(m_k)$ and is thus a polynomial generator
of $\Omega^U_{2k}\otimes \bZ_{(p)}$.

\textbf{Proof at the prime $2$.} Everything is now localized at $2$,
and we use the notation introduced before Lemma \ref{lem:ku}.
We show that in $\Omega_*^{\spinc}\otimes\bZ_{(2)}$,
\begin{equation}\label{eq:2local}
\ker\alpha^c
=\im T_{\MSpin}+(e_1,e_2,e_3,\ldots)
=\im T_{\MSpin}+\im T_{\MSpinc},
\end{equation}
where $(e_1,e_2,\ldots)$ denotes the ideal generated by the classes
\eqref{eq:ej}.  We have already seen that the right-hand side is
contained in $\ker\alpha^c$.
Since $\im T_{\MSpin}\subseteq\ker\alpha^c$, by \eqref{eq:quotient} the
orientation $\alpha^c$ induces a homomorphism
\[
\overline\alpha\co\pi_*\cC\longrightarrow ku_* .
\]
Multiplicativity of $\alpha^c$, the vanishing of $\alpha^c(e_j)$ for
$j\ge 1$, and \eqref{eq:section} give
\[
\overline\alpha\circ(f_0)_*=\mathrm{id},\qquad
\overline\alpha\circ(f_j)_*=0\quad(j\ge1).
\]
For $z\in\ker\alpha^c\cap\Omega_n^{\spinc}$, Proposition~\ref{prop:Phi}
expresses $q_*z$ as $\sum_{j\ge 0}(f_j)_*(a_j)$ with
$a_j\in ku_{n-4j}$; applying $\overline\alpha$ gives $a_0=0$, so that
\[
q_*z=\sum_{j\ge1}(f_j)_*(a_j)
=\sum_{j\ge1}q_*\bigl(e_j\,s_*(a_j)\bigr),
\qquad a_j\in ku_{n-4j}.
\]
The sum is finite, and
\[
z-\sum_{j\ge1}e_j\,s_*(a_j)\in\ker q_*=\im T_{\MSpin}
\]
by \eqref{eq:quotient}. Hence
$z\in\im T_{\MSpin}+(e_1,e_2,\ldots)$.
Since each $e_j$ lies in $\im T_{\MSpinc}$ and that image is an ideal
by Lemma \ref{lem:T}(b), the asserted equalities \eqref{eq:2local}
follow, and the proof of Theorem \ref{thm:transfer} is complete.
\end{proof}

\begin{remark}
\label{rem:bottomclasses}
The use of the whole family of bottom classes $e_j$ in the proof at the
prime $2$ is substantive. One might be tempted to argue that the
transfer map ``acts the same way on all the summands'' in
\eqref{eq:kowedgeCP}, i.e., to replace the classes $e_j$ by the powers
$e_1^j$ and thus to reduce to the study of the transfer map on the
bottom summand $\ku$.  (This was the argument given in the published
version of this paper.)  However, \eqref{eq:kowedgeCP} is not a
splitting of $\ku$-module spectra, as was shown in
\cite[\S7]{MR1259520} (the argument there was about the
Anderson--Brown--Peterson splitting of $\MSpin$ but would also apply to
$\MSpinc$), and in fact the argument cannot work: in degree $8$,
\[
\langle c_1^4,e_1^2\rangle=6,\qquad
\langle c_1^4,e_2\rangle=\pm1
\]
by Lemma \ref{lem:geometry}.
Indeed every element of $e_1\cdot\Omega_4^{\spinc}$ has even $c_1^4$:
on $e_1\cdot[B,L_B]$ this number is
$6\langle c_1(L_B)^2,[B]\rangle$.
Elements in the image of $T_{\MSpin}$ have zero top power of the
determinant class, since the determinant is pulled back from a
lower-dimensional base. Thus $e_2$ cannot be obtained from $e_1$ alone,
even after collapsing out the image of the transfer. An additive splitting does not
justify replacing the classes $e_j$ by repeated multiplication by $e_1$.
\end{remark}

\bibliographystyle{amsplain}

\bibliography{PSCFiberedSing-v2}

\end{document}